\documentclass[10pt]{amsart}

\usepackage{graphicx} 
\usepackage[colorlinks=false]{hyperref} 
\usepackage{amssymb} 
\usepackage{amsfonts} 
\usepackage{amsmath,amscd} 
\usepackage{amsthm}  
\usepackage{mathtools} 
\usepackage{csquotes} 
\usepackage{tikz-cd} 
\usepackage{wrapfig} 
\usepackage{tabularx} 

\usepackage{enumerate}
\usepackage{xcolor}
\usepackage{mathrsfs}
\usepackage[makeroom]{cancel}
\usepackage{float}
\usepackage{array}

\newcounter{dummy} \numberwithin{dummy}{section}
\newtheorem{theorem}[dummy]{Theorem}
\newtheorem{corollary}[dummy]{Corollary}

\newtheorem{lemma}[dummy]{Lemma}
\newtheorem{definition}[dummy]{Definition}
\newtheorem{proposition}[dummy]{Proposition}
\theoremstyle{remark}
\newtheorem{remark}[dummy]{Remark}
\newtheorem{example}[dummy]{Example}

\DeclareMathOperator{\id}{id}

\DeclareMathOperator{\pr}{pr}
\DeclareMathOperator{\rank}{rank}
\DeclareMathOperator{\spn}{span}
\DeclareMathOperator{\sgn}{sgn}

\DeclareMathOperator{\Ad}{Ad}

\DeclareMathOperator{\SO}{SO}

\DeclareMathOperator{\Ort}{O}
\DeclareMathOperator{\so}{\mathfrak{so}}

\DeclareMathOperator{\Isom}{Isom}
\DeclareMathOperator{\Hol}{Hol}

\DeclareMathOperator{\nil}{\mathfrak{nil}}
\DeclareMathOperator{\Nil}{Nil}

\numberwithin{equation}{section}

\title[On $\mathrm{G}_2$ and Model Spaces of Step and Rank Three]{On $\mathrm{G}_2$ and Sub-Riemannian Model Spaces of Step and Rank Three}

\author{Eirik Berge and Erlend Grong}
\date{}

\address{Department of Mathematical Sciences, Norwegian University of Science and Technology, 7491 Trondheim, Norway.}
\email{eirik.berge@ntnu.no}

\address{University of Bergen, Department of Mathematics, P. O. Box 7803,
5020 Bergen, Norway.}
\email{erlend.grong@uib.no}

\thanks{The second author is supported by the Research Council of Norway (project number 249980/F20).
 The authors were partially supported by the joint NFR-DAAD project 267630/F10.
Results are partially based on the first author's Master Thesis at the University of Bergen, Norway.}

\subjclass[2010]{53C17}

\keywords{Sub-Riemannian geometries, model spaces, isometries, $\mathrm{G}_2$}

\begin{document}

\begin{abstract}
We give the complete classification of all sub-Riemannian model spaces with both step and rank three. They will be divided into three families based on their nilpotentization. Each family will depend on a different number of parameters, making the result crucially different from the known case of step two model spaces. In particular, there are no nontrivial sub-Riemannian model spaces of step and rank three with free nilpotentization. We also realize both the compact real form $\mathfrak{g}_2^c$ and the split real form $\mathfrak{g}_2^s$ of the exceptional Lie algebra $\mathfrak{g}_2$ as isometry algebras of different model spaces.
\end{abstract}

\maketitle

\section{Introduction}

The development of Riemannian geometry has been highly influenced by certain spaces with maximal symmetry called \textit{model spaces}. Their ubiquity presents itself throughout differential geometry from the classical Gaussian map for surfaces to comparison theorems based on volume, the Laplacian, or Jacobi fields \cite{Peter_Petersen}. Following in the footsteps of Klein's Erlangen program, model spaces fit with the approach of investigating the symmetries of a geometric object to understand the object itself. In the Riemannian setting, work by Wilhelm Killing and Heinz Hopf among others resulted in the complete classification of the Riemannian model spaces in the celebrated Killing-Hopf Theorem. The theorem states that the only model spaces in Riemannian geometry are the spheres, the hyperbolic spaces, and the Euclidean spaces with their standard structures. This result is not only influential but also remarkable due to the fact that all the model spaces were already known and well examined prior to the classification. \par
In recent years sub-Riemannian geometry has emerged as an active field of research with ties to optimal control theory, Hamiltonian mechanics, geometric measure theory, and harmonic analysis. Nevertheless, new results in this subject are often only derived for special classes such as Carnot groups and contact structures. This is not due to a lack of ability, but rather that the absence of a canonical connection as present in Riemannian geometry has complicated issues. Thus it is of interest to enlarge the concept of model spaces to the sub-Riemannian setting to establish reference spaces. \par
\emph{A sub-Riemannian model space} as defined in \cite{sub_Riemannian_Model_Spaces} is a simply connected and bracket generating sub-Riemannian manifold $(M,H,g)$ satisfying the following \textit{integrability condition}: For any points $x,y \in M$ and any linear isometry $q:H_x \to H_y$ there exists a smooth isometry $\varphi:M \to M$ such that $d\varphi|_{H_x} = q$. These spaces have a canonical partial connection on $H$ whose holonomy is either trivial or isomorphic to $\SO(\rank (H))$. All model spaces of with $H + [H, H] = TM$ were classified in \cite{sub_Riemannian_Model_Spaces}, as well as the case of step three and $\rank (H) =2$ model spaces. In all these cases there is only one possible nilpotentization, namely the free nilpotent Lie group of appropriate step and rank. The main goal of the paper is to obtain some insight into model spaces with higher steps by classifying all sub-Riemannian model spaces of step and rank three. Our main results can be summarized as follows. 

\begin{theorem}
\label{introductory_theorem}
The sub-Riemannian model spaces with step and rank three have dimension $9$, $11$, or $14$. 
\begin{enumerate}[\rm (a)]
\item In dimension 9, there is a two-parameter family $\mathsf{C}_{3,3}(a_1, a_2)$, $a_1, a_2 \in \mathbb{R}$ of model spaces, behaving under scaling of the metric like \[\lambda \mathsf{C}_{3,3}(a_1, a_2) = \mathsf{C}_{3,3}(a_1/\lambda^2, a_2/\lambda^4), \quad \lambda > 0.\] These spaces can be made into Lie groups with left invariant sub-Riemannian structures such that the orientation preserving isometries are given by compositions of left translations and Lie group automorphisms if and only if $a_2 \geq 0$ or if $a_2 < 0$ and $a_1 \leq 2\sqrt{-a_2}$. Moreover, they are compact precisely when $a_2 < 0$ and $2\sqrt{-a_2} < a_1$.
\item In dimension 11, there are three model spaces up to scaling. One of these model spaces is a Carnot group, while the two other model spaces have respectively the compact real form $\mathfrak{g}_2^c$ and the split real form $\mathfrak{g}_2^s$ of the exceptional Lie algebra~$\mathfrak{g}_2$ as their isometry algebras.
\item In dimension 14, the only model space is the free nilpotent Lie group $\mathsf{F}[3,3]$ of step and rank three.
\end{enumerate}
\end{theorem}

In particular, all sub-Riemannian model spaces of step and rank three with the same infinitesimal structure form respectively a two-parameter family, a one-parameter family, or are represented by a single space. Furthermore, we give the first realization of the exceptional Lie group $\mathrm{G}_2$ not only as an algebra of symmetries, see e.g. \cite{BoMo09} and references therein, but as an isometry group. This result also shows that the model spaces of step and rank three have unique structures that can not be observed in any other case. \par
The structure of the paper is as follows: In Section \ref{sec:Preliminaries} we introduce notation and briefly review standard material from sub-Riemannian geometry. The results about sub-Riemannian model spaces that we will need are given in Section \ref{Sec: SubGeo}, mostly taken from \cite{sub_Riemannian_Model_Spaces}. With the exception of Example \ref{maximal_dimension_example}, everything in these sections is previously known. In Section \ref{sec: Model_spaces_classification} we determine the sub-Riemannian model spaces of step and rank three that are also Carnot groups. The cases in Theorem~\ref{introductory_theorem} (a), (b) and (c) are treated respectively in Sections \ref{Sec: C33}, \ref{sec:(3,6,11)-classification} and \ref{sec:Section (3,6,14)-classification}, where we give a detailed classification with explicit isometry groups and sub-Riemannian structures. Just as three-dimensional Riemannian spheres have invariant Lie group structures, there are several sub-Riemannian model spaces that have Lie group structures in rank three. This holds even though the holonomy of the model space analogue of the Levi-Civita connection is non-zero. We show several examples of this in Section \ref{Sec: C33}. Finally, there are sub-Riemannian structures on both the compact real form $\mathrm{G}_2^c$ and the split real form $\mathrm{G}_2^s$ of $\mathrm{G}_2$ where the isometry group has maximal dimension, see Remark~\ref{remark:G2SR}. However, since they are not symmetric spaces they do not fall under our definition of a sub-Riemannian model space.

\section{Preliminaries}
\label{sec:Preliminaries}

\subsection{Sub-Riemannian Geometry}
In this section we will recall basic definitions and terminology in sub-Riemannian geometry. \emph{A sub-Riemannian manifold} is a triple $(M,H,g)$ where $M$ is a connected manifold, $H \subseteq TM$ is a subbundle, and $g$ is a smooth fiber metric defined on $H$. We refer to $H$ as \emph{the horizontal distribution} and to $g$ as \emph{the sub-Riemannian metric}. From~$H$, we obtain a flag of subsheaves \[\underline{H} \subseteq \underline{H}^2 \subseteq \cdots \subseteq \underline{H}^{j} \subseteq \cdots \subseteq \Gamma(TM), \qquad \underline{H} := \Gamma(H), \quad \underline{H}^j := \underline{H}^{j-1} + [\underline{H}, \underline{H}^{j-1}].\]
The notation $H_x^{j}$ for $x \in M$ will be used for the subspace of $T_x M$ consisting of the elements $X(x)$ where $X \in \underline{H}^{j}$. 
We say that $H$ is \textit{bracket-generating} if for every $x \in M$ there is a minimal number $r(x)$, called \emph{the step of} $H$, such that \[H_{x}^{r(x)} = T_{x}M.\] By using the abbreviation $n_{i}(x) := \textrm{rank}(H_{x}^{i})$ we form the multi-index
\[\mathfrak{G}(x) := (n_{1}(x), \dots, n_{r(x)}(x)),\]
called \emph{the growth vector} at $x \in M$. We call the subbundle $H$ \textit{equiregular} whenever the growth vector does not depend on the point $x \in M$. For a bracket-generating and equiregular horizontal distribution $H$ we obtain a flag of subbundles
\[H \subseteq H^2 \subseteq \dots \subseteq H^{r} = TM, \qquad H^i|_x= H_{x}^{i}. \]

An absolutely continuous curve $\gamma:[a,b] \to M$ will be called \emph{horizontal} if $\dot{\gamma}(t) \in H_{\gamma(t)}$ for almost every $t \in [a,b]$. We define the length of a horizontal curve $\gamma$ analogously as in the Riemannian setting by \[L(\gamma) = \int_{a}^{b} | \dot{\gamma}(t) |_g \,dt.\]
Associated to any sub-Riemannian manifold $(M,H,g)$ with $H$ bracket-generating is a distance function $d_{CC}$ called \emph{the Carnot-Carath\'eodory distance}. It is defined by $d_{CC}(x,y) = \inf L(\gamma)$, where the infimum is taken over all horizontal curves connecting the points $x$ and $y$ in $M$. Since the latter set of curves is non-empty by the Chow-Rashevskii Theorem \cite{Chow1940,Ras38}, the Carnot-Carath\'eodory distance gives $M$ the structure of a metric space. 

\subsection{Nilpotentization Procedure}
\label{subsec:nilpotentization_procedure}
A Lie algebra $\mathfrak{n}$ is called \emph{stratifiable} with \textit{step} $r$ if it can be decomposed as
\[\mathfrak{n} = \mathfrak{n}_1 \oplus \cdots \oplus \mathfrak{n}_r, \qquad [\mathfrak{n}_1, \mathfrak{n}_j] = \mathfrak{n}_{j+1}, \quad j=1, \dots, r-1, \qquad [\mathfrak{n}_1, \mathfrak{n}_r] = 0.\]
Such a Lie algebra is necessarily nilpotent and we call $\mathfrak{n}_1$ the \textit{generating layer} for $\mathfrak{n}$. \emph{A Carnot group} $(N, E, h)$ is a connected and simply connected Lie group $N$ whose Lie algebra $\mathfrak{n}$ is stratifiable, together with a sub-Riemannian structure $(E,h)$ defined by left translation of the generating layer $\mathfrak{n}_1$ with an inner product.

If $(M,H,g)$ is an equiregular sub-Riemannian manifold of step $r$, we can associate a Carnot group to each point of $M$ as follows: Fix a point $x \in M$ and introduce the stratified Lie algebra
\[\nil(x) = H_{x} \oplus H_{x}^{2}/H_x \oplus H_x^{3}/H_x^2 \oplus \dots \oplus T_xM/H_x^{r-1},\]
with brackets defined by
\[[[v],[w]] := [X,Y](x) \, \, \textrm{mod} \, H_{x}^{i + j - 1}, \qquad \quad [v] \in H_x^{i}/H_x^{i-1}, \, [w] \in H_x^{j}/H_x^{j-1},\]
where $X \in \underline{H}^{i}$ and $Y \in \underline{H}^j$ satisfy $X(x) = v$ and $Y(x) = w$. We refer the reader to \cite[Proposition 4.10]{Montgomery} for details about why this bracket is well defined. Let $\Nil(M,x)$ denote the connected and simply connected Lie group corresponding to $\mathfrak{nil}(x)$, equipped with a sub-Riemannian structure defined by left translation of $H_x$ and its inner product induced by $g$. This construction coincides with Gromov's description of the tangent cone of a sub-Riemannian manifold, see \cite[Chapter 8.4]{Montgomery} and the references within. 

\section{Sub-Riemannian model spaces}
\label{Sec: SubGeo}
\subsection{Isometries and Model Spaces}

This section will give the definition and basic properties of sub-Riemannian model spaces. All results together with justifications can be found in \cite{sub_Riemannian_Model_Spaces}.
\begin{definition}
\label{model_space_definition}
A (sub-Riemannian) model space is a simply connected and bracket-generating sub-Riemannian manifold $(M,H,g)$ satisfying the following integrability condition: For any points $x,y \in M$ and any linear isometry $q:H_x \to H_y$ there exists a smooth isometry $\varphi:M \to M$ such that $d\varphi|_{H_x} = q$.
\end{definition}
It follows from \cite[Theorem~8.17]{Semi-Riemannian} that Definition \ref{model_space_definition} generalizes the traditional model spaces in Riemannian geometry. Let us clarify the notion of isometry in the above definition: An \textit{isometry} between bracket generating sub-Riemannian manifolds $(M,H,g)$ and $(N,E,h)$ is a homeomorphism $\Phi:M \to N$ that preserves the Carnot-Carath\'eodory distances. The following regularity result about isometries is assembled from \cite[Theorem~1.2, Theorem~1.6, Corollary~1.8]{Smoothness_of_sub_riemannian_isometries}.

\begin{proposition}
\label{Regularity_of_isometries}
Let $(M,H,g)$ be a sub-Riemannian manifold that is bracket-generating and equiregular. Then any isometry $\varphi:M \to M$ is a smooth map satisfying $d\varphi (H) \subseteq H$ and restricts to a linear isometry \[d\varphi|_{H_x}:H_x \longrightarrow H_{\varphi(x)},\] for any $x \in M$. Moreover, any isometry is uniquely determined by its restricted differential $d\varphi|_{H_x}$ at a single point $x \in M$. The collection of all isometries $\Isom(M)$ from $M$ to itself forms a Lie group under composition.
\end{proposition}

We note that if $(M, H, g)$ is a bracket-generating and equiregular sub-Riemannian manifold with $\dim(M) = m$ and $\rank(H) = n$, we have \[\dim\left( \Isom(M)\right) \leq m + \frac{n(n-1)}{2}\] by Proposition~\ref{Regularity_of_isometries}, with equality if $M$ is a model space. In the Riemannian setting, the assumption that the dimension of $\Isom(M)$ is maximal is sufficient for a simply connected Riemannian manifold $M$ to be a model space by \cite[Theorem 6.3.3]{Kobayashi_Numisu}. However, the following example shows that this is not the case for simply connected sub-Riemannian manifolds.

\begin{example}[Maximal dimension and model spaces]
\label{maximal_dimension_example}
Let $\mathfrak{n} = \mathfrak{n}_1 \oplus \mathfrak{n}_2$ be the stratified Lie algebra with generating layer $\mathfrak{n}_1 =\spn \{ X_1, X_2, X_3, X_4 \}$, with center $\mathfrak{n}_2 = \spn \{ Y_1, Y_2, Y_3 \}$, and with bracket relations
$$[X_1, X_2] = [X_3, X_4] = Y_1, \quad [X_1, X_3] = [X_4, X_2] = Y_2, \quad [X_1 , X_4] = [X_2, X_3] = Y_3.$$
We give $\mathfrak{n}_1$ an inner product by requiring that $X_1, X_2, X_3, X_4$ form an orthonormal basis for $\mathfrak{n}_1$.
Let $N$ be the connected and simply connected Lie group corresponding to $\mathfrak{n}$ with sub-Riemannian structure $(E,h)$ defined by left translation of $\mathfrak{n}_1$ and its inner product. All the left translations are in $\Isom(N)$ by definition. Furthermore, it can be verified that for any orientation preserving isometry $q :\mathfrak{n}_1 \to \mathfrak{n}_1$, we can find an isometry $\varphi: N \to N$ with $\varphi(1) = 1$ and $\varphi |_{\mathfrak{n}_1} = q$. Hence we have that $\Isom(N)$ has the maximal dimension \[13 = \dim(N) + (\rank(E))(\rank(E) -1)/2.\] However, orientation reversing isometries can not be integrated because it follows from \cite[Example 4.1]{sub_Riemannian_Model_Spaces} that the free nilpotent Lie group $\mathsf{F}[4,2]$ is the only model space with step two and rank four that is a Carnot group. We will define the free nilpotent Lie group $\mathsf{F}[n,r]$ of step $n$ and rank $r$ in Example \ref{free_lie_algebra}.
\end{example}

\subsection{Canonical Partial Connections} \label{sec: Partial_Connections_and_Horizontal_Holonomy}
We recall that a partial connection $\nabla$ on a vector bundle $E \to M$ in the direction of $H \subseteq TM$ is a map 
$$\nabla:\Gamma(H) \times \Gamma(E) \longrightarrow \Gamma(E) , \qquad  (X,Y)  \longmapsto \nabla_{X}Y,$$
that is $C^{\infty}(M)$-linear in the first component, $\mathbb{R}$-linear in the second, and satisfies the Leibniz property \[\nabla_{X}\left(fY\right) = X(f)Y + f\nabla_{X}Y, \qquad X \in \Gamma(H), \, Y \in \Gamma(E), \, f \in C^{\infty}(M).\] Similarly as for regular affine connections, we can define a holonomy group $\Hol^\nabla(x)$ at a point $x \in M$ by considering parallel transport along loops based at $x$ and horizontal to $H$. If $H$ is bracket-generating then $\Hol^\nabla(x)$ is a Lie group and it will be connected whenever $M$ is simply connected. More details on holonomy of partial connections can be found in \cite{CGJK15}. \par 
The partial connections we will be interested in are on the horizontal bundle $H$ of a sub-Riemannian manifold with metric $g$. We say that a partial connection $\nabla$ on $H$ in the direction of $H$ is \textit{compatible with $g$} if
\[X \langle Y,Z \rangle_g = \langle \nabla_{X}Y , Z \rangle_g + \langle Y,\nabla_{X} Z \rangle_g, \qquad X,Y,Z \in \underline{H}. \]
If the partial connection $\nabla$ satisfies 
\begin{equation} \label{invariant} \nabla_{\Ad(\varphi)X} \Ad(\varphi)Y = \Ad(\varphi)\nabla_{X} Y, \quad \Ad(\varphi)X(x):=d\varphi\circ X\circ \varphi^{-1}(x),\end{equation}
for any $\varphi \in \Isom(M)$ we call it \textit{invariant under isometries}.

Let $(M,H,g)$ denote a sub-Riemannian model space and let us abbreviate $G := \Isom(M)$ and $\mathfrak{g}:=\textrm{Lie}(G)$. We denote by \[K_x := \left\{\Phi \in G \, \Big | \, \Phi(x) = x \right\}\] the isotropy group corresponding to the point $x \in M$. The isotropy groups are all conjugate to each other and are compact by \cite[Corollary 5.6]{Polish_Metric_Spaces}. We omit the explicit reference to the point $x$ in the notations $K := K_x$ and $\Hol^\nabla := \Hol^\nabla(x)$, as this is of minor relevance in our arguments. Hence we can consider the principal bundle
\begin{equation} \label{PrincipalG} K \longrightarrow{} G \xrightarrow{\, \, \pi \, \, } M = G/K, \nonumber \end{equation}
where $\pi(\varphi) = \varphi(x)$ for $\varphi \in G$. Recall that the isotropy group $K$ acts on $\mathfrak{g}$ by restricting the adjoint action of $G$ on $\mathfrak{g}$. Moreover, we use the notation $\ell_{g}:G \to G$ for the map sending $h \in G$ to $gh$.

\begin{proposition} \label{canonical_connection_theorem}
\begin{enumerate}[\rm (a)]
\item There exists a unique partial connection $\nabla$ on $H$ in the direction of $H$ that is compatible with $g$ and invariant under isometries.
\item There is a unique subspace $\mathfrak{p}$ of $\mathfrak{g}$ that is $K$-invariant and mapped bijectively onto $H_x$ by $d\pi$. Let $\gamma: [a,b] \to M$ be a horizontal curve and let $\varphi: [a,b] \to G$ be a curve above $\gamma$, that is, $\pi(\varphi(t)) = \gamma(t)$ for every $t \in [a,b]$. Assume that \[d\ell_{\varphi(t)^{-1}}\left(\dot \varphi(t)\right) \in \mathfrak{p}\] for every $t \in [a,b]$. Then
$$X(t) = d\varphi(t) d\varphi(a)^{-1} v$$
is a $\nabla$-parallel vector field along $\gamma$ for any $v \in H_{\gamma(a)}$.
\item The holonomy group $\Hol^\nabla$ does not depend on the chosen base point $x \in M$. It is either trivial or isomorphic to $\SO(n)$, where $n$ is the rank of $H$. If $\Hol^\nabla$ is trivial, then $M$ can be given a Lie group structure such that $(H,g)$ is left-invariant. Furthermore, every isometry is then a composition of a left translation and a Lie group automorphism.
\item Define an inner product $\langle \, \cdot \, , \, \cdot \, \rangle_{\mathfrak{p}}$ on $\mathfrak{p}$ by
$$\langle A , B\rangle_{\mathfrak{p}} = \langle d\pi \, A , d\pi \, B\rangle_{g}.$$
Assume $\Hol^\nabla$ is isomorphic to $\SO(n)$ and consider $(G,\tilde H, \tilde g)$ with $(\tilde H, \tilde g)$ defined by left translation of $\left( \mathfrak{p} , \langle \, \cdot \, , \, \cdot \, \rangle_{\mathfrak{p}} \right)$. Then $(G,\tilde H, \tilde g)$ satisfies the definition of a sub-Riemannian model space with the possible exception of being connected and simply connected. This can easily be remedied by considering the universal cover of the identity component of $G$ together with the lifted structure.
\end{enumerate}
\end{proposition}
We will use the term \emph{canonical partial connection} for the partial connection in Proposition~\ref{canonical_connection_theorem}~(a). The canonical partial connection gives us a necessary condition for sub-Riemannian model spaces $(M,H, g)$ and $(\hat M, \hat H, \hat g)$ to be isometric as we now explain.
Assume $\Phi:M \to \hat M$ is an isometry. Choose a point $x \in M$ and consider the isotropy groups
\[K := K_{x} \subseteq G := \Isom(M), \qquad \hat K := \hat K_{\Phi(x)} \subseteq \hat G := \Isom(\hat M).\]
\begin{wrapfigure}{r}{5cm}
\begin{tikzcd}[row sep=scriptsize, column sep=scriptsize]
& \mathfrak{k} \arrow[dl,"\exp"] \arrow[rrr,dashed] \arrow[dd,hookrightarrow] & & & \hat{\mathfrak{k}} \arrow[dd,hookrightarrow] \arrow[dl,"\exp"] \\
K \arrow[dd] \arrow[rrr,"\bar{\Phi}"] & & & \hat K \arrow[dd]\\
& \mathfrak{g} \arrow[dl,"\exp"] \arrow[rrr,dashed] & & & \hat{\mathfrak{g}} \arrow[dl,"\exp"] \\
G \arrow[dd,twoheadrightarrow] \arrow[rrr,"\bar{\Phi}"] & & & \hat{G} \arrow[dd,twoheadrightarrow]\\
& \mathfrak{p} \arrow[uu,hookrightarrow] \arrow[rrr,dashed] & & & \hat{\mathfrak{p}} \arrow[uu,hookrightarrow] \\
 M \arrow[ur,rightarrowtail,"\nabla"] \arrow[rrr,"\Phi"] & & & \hat M \arrow[ur,rightarrowtail,"\hat \nabla"]
\end{tikzcd}
\end{wrapfigure}
We let $\mathfrak{k}$ and $\mathfrak{g}$ denote the Lie algebras of respectively $K$ and $G$, and similarly let $\hat{ \mathfrak{k}}$ and $\hat {\mathfrak{g}}$ be Lie algebras of $\hat K$ and $\hat G$.
The map $\Phi$ induces a group homomorphism $\bar{\Phi}:G \longrightarrow \hat G$ by conjugation $\bar{\Phi}(\varphi) = \Phi \circ \varphi \circ \Phi^{-1}$ for $\varphi \in G.$
Note that $\bar{\Phi}(K) = \hat K$ and that $d\bar{\Phi}$ restricts to a Lie algebra isomorphism from $\mathfrak{k}$ onto $\hat{\mathfrak{k}}$. Let $\mathfrak{p} \subseteq \mathfrak{g}$ be the canonical subspace described in Proposition~\ref{canonical_connection_theorem}~(b) with the inner product given in Proposition~\ref{canonical_connection_theorem}~(d). The subspace $\mathfrak{p}$ should be thought of as a principal bundle analogue of the canonical partial connection $\nabla$. Define $\hat {\mathfrak{p}}$ similarly as a subspace of $\hat{\mathfrak{g}}$ corresponding to canonical connection $\hat \nabla$. Since $\Phi$ takes $\nabla$-parallel vector fields along horizontal curves to $\hat \nabla$-parallel vector fields, we have that $d\Phi (\mathfrak{p}) \subseteq \hat {\mathfrak{p}}$. Furthermore, $d\Phi$ restricted to $\mathfrak{p}$ is an isometry since $d\Phi$ maps $H_x$ isometrically onto $\hat{H}_{\Phi(x)}$. These observations will be used in the proof of Lemma \ref{A_33_Classification}.

\subsection{Carnot Model Spaces and Free Nilpotent Lie Groups}

We will refer to the model spaces that are also Carnot groups as \textit{Carnot model spaces} for convenience. The different nilpotentizations $\Nil(M,x)$ of a sub-Riemannian model space $(M,H,g)$ does not depend on the point $x \in M$. Hence we will simplify the notation $\Nil(M,x)$ to $\Nil(M)$ whenever $(M,H,g)$ is a model space. The following proposition reveals that Carnot model spaces act as an invariant for classifying all model spaces.

\begin{proposition}
\label{Nilpotentization_Model_Space}
If $(M,H,g)$ is a sub-Riemannian model space then $\Nil(M)$ is a Carnot model space with the same growth vector. In particular, the only growth vectors sub-Riemannian model spaces can have are those that occur on Carnot model spaces.
\end{proposition}

This gives a strategy for classifying model spaces: We start by classifying the Carnot model spaces of a certain step and rank. Then we will divide all model spaces with the same step and rank into distinct families according to their nilpotentization. The following example is of particular importance for classifying all Carnot model spaces. 
\begin{example}
\label{free_lie_algebra}
Fix a set $\{X_1,\dots,X_n\}$ of $n$ elements. Consider the free vector space of all elements formally of the form \[[X_{i_1},[X_{i_2},\dots,[X_{i_{k-1}},X_{i_{k}}]\dots]],\] where $i_j \in \{1,\dots,n\}$ for every $j = 1,\dots, k$. We identify two such elements if they can be transformed into one another via either bilinearity, skew-symmetry, or the Jacobi identity of the formal bracket $[\cdot,\cdot]$. The resulting space has a natural bracket operation turning it into a Lie algebra called the \textit{free Lie algebra} of rank $n$. The free Lie algebra only depends, up to isomorphism, on the cardinality of the generating set and not the choice of elements themselves. It has an obvious grading with the generating layer spanned by $X_1, \dots, X_n$. \par 
By identifying each element consisting of $r+1$ or more brackets with zero, we obtain the \textit{free nilpotent Lie algebra} $\mathsf{f}[n,r]$ of rank $n$ and step $r$. This inherits a grading $\mathsf{f}[n,r] = \mathfrak{f}_1 \oplus \cdots \oplus \mathfrak{f}_r$ from the free Lie algebra. The dimension of the $k$'th layer is given by \textit{Witt's formula} 
\begin{equation}
\label{Witts_Formula}
    \textrm{dim}(\mathfrak{f}_{k}) = \frac{1}{k}\sum_{d|k}\mu(d)n^{k/d}, \quad k \leq r,
\end{equation}
where $\mu$ denotes the \textit{M\"obius function}. Let $\mathsf{F}[n,r]$ denote the connected and simply connected Lie group corresponding to $\mathsf{f}[n,r]$ called the \textit{free nilpotent Lie group} of rank $n$ and step $r$. Fix an inner product $\langle \cdot, \cdot \rangle$ on the generating layer $\mathfrak{f}_1$ of $\mathsf{f}[n,r]$. We equip $\mathsf{F}[n,r]$ with the left-translated structure $(H,g)$ of the generating layer $\mathfrak{f}_1$ and its inner product $\langle \cdot, \cdot \rangle$. It follows from \cite{sub_Riemannian_Model_Spaces} that $(\mathsf{F}[n,r],H,g)$ is a Carnot model space. 
\end{example}

The classification of all Carnot model spaces can be reformulated as the following representation theory problem on the free nilpotent Lie algebras. 

\begin{proposition}
\label{strategy_carnot_groups}
Let $\Ort(\mathfrak{f}_1)$ denote the linear isometries of the generating layer $\mathfrak{f}_1$ of $\mathsf{f}[n,r]$. Given $q \in \Ort(\mathfrak{f}_1)$, we write \[\varphi_q: \mathsf{F}[n,r] \to \mathsf{F}[n,r]\] for the unique sub-Riemannian isometry satisfying $d\varphi_q |\mathfrak{f}_1 = q$. Consider $\mathsf{f}[n,r]$ as a representation of $\Ort(\mathfrak{f}_1)$ through the map $q \mapsto d_1 \varphi_q$. Let $\mathfrak{a}$ be an ideal in $\mathsf{f}[n,r]$ that is also a sub-representation and assume furthermore that $\mathfrak{f}_r \not \subseteq \mathfrak{a}$. Then
$$\mathfrak{n} = \mathsf{f}[n,r]/\mathfrak{a}$$
is the Lie algebra of a Carnot model space of rank $n$ and step $r$. Moreover, all Lie algebras of Carnot model spaces can be obtained in this way.
\end{proposition}
We note that $d_1 \varphi_q$ is uniquely determined by the relations
$$d_1 \varphi_q [B_1 , B_2] =  [ d_1 \varphi_q B_1 , d_1 \varphi_q B_2], \qquad d_1 \varphi_q A = q A, \qquad B_1, B_2 \in \mathfrak{f}, \, A \in \mathfrak{f}_1.$$

\subsection{Rank Three Model Spaces and Compatible Lie Group Structures} \label{sec:CrossProduct}
Let $(M,H,g)$ be a sub-Riemannian model space where $H$ has rank three and let $\nabla$ be the canonical partial connection on $H$. Recall that Proposition~\ref{canonical_connection_theorem}~(c) gives a relation between the holonomy of $\nabla$ and compatible Lie group structures. However, for model spaces with rank three there is an additional relation between holonomy and compatible Lie group structures, which we will now describe.

One can easily show that the horizontal bundle $H$ of any model space is orientable. Any choice of orientation gives a corresponding cross product $\times: \wedge^2 H \to H$ on $H$ whenever the rank of $H$ is three. The choice of orientation only determines the sign of the cross product. Therefore, the collection of maps $v \wedge w\mapsto c (v \times w)$, $c \in \mathbb{R}$ does not depend on the choice of orientation.
Let $\Isom_0(M)$ denote the identity component of $\Isom(M)$. We say that a partial connection $\tilde \nabla$ is \emph{invariant under orientation-preserving isometries} if it satisfies \eqref{invariant} for any $\varphi \in \Isom_0(M)$. 
\begin{lemma} \label{lemma:CrossProduct}
Let $(M,H,g)$ be a sub-Riemannian model space with rank three. Consider a choice of cross product $\times$ on $H$, and denote by $\nabla$ the canonical partial connection on $H$. For a given point $x \in M$, define $K_0 \subseteq \Isom_0(M)$ as the isotropy group corresponding to $x$. Write the Lie algebras of $\Isom_0(M)$ and $K_0$ as respectively $\mathfrak{g}$ and $\mathfrak{k}$. Let $\mathfrak{p}$ be the subspace of $\mathfrak{g}$ corresponding to $\nabla$ as described in Proposition~\ref{canonical_connection_theorem}~{\rm (b)}. 
\begin{enumerate}[\rm (a)]
\item Any partial connection that is compatible with $g$ and invariant under orientation-preserving isometries is on the form
\begin{equation} \label{nablac} \nabla_X^c Y = \nabla_X Y + c X \times Y, \qquad X, Y \in \underline{H},\end{equation}
for some $c \in \mathbb{R}$.
\item Relative to our choice of cross product, there is a unique linear map $L: \mathfrak{p} \to \mathfrak{k}$ with the property
$$\pi_* [L(A), B] = \pi_* A \times \pi_* B, \quad A, B \in \mathfrak{p}.$$
Furthermore, this map is equivariant under the action of $K_0$. 
\item Define $\mathfrak{p}_c := \{ A + cL(A) \, | \, A \in \mathfrak{p} \}$. Assume that $\gamma: [t_0,t_1] \to M$ is a horizontal curve and $\varphi: [t_0,t_1] \to \Isom_0(M)$ is a curve above $\gamma$ such that \[d\ell_{\varphi(t)^{-1}}\left(\dot \varphi(t)\right) \in \mathfrak{p}_{c},\] for all $t \in [a,b]$. Then
\[X(t) = d\varphi(t) d\varphi(t_0)^{-1} v\]
is a $\nabla^c$-parallel vector field along $\gamma$ for any $v \in H_{\gamma(a)}$.
\item Fix $c \in \mathbb{R}$. The holonomy group of $\nabla^c$ is either trivial or isomorphic to $\SO(3)$. If the holonomy group is trivial, then $M$ has a Lie group structure with identity $x$ such that $(H,g)$ is left-invariant. Moreover, every orientation-preserving isometry is then a composition of a left-translation and a Lie group automorphism.
\end{enumerate}
\end{lemma}
The proof of this result is given in \cite[Proposition~3.8]{sub_Riemannian_Model_Spaces}. Whenever the rank of $H$ is not equal to three, the canonical connection is the only compatible connection that is invariant under orientation-preserving isometries. In the Riemannian case, this difference is reflected in the fact that the three-dimensional spheres have invariant Lie group structures.

\section{Carnot Model Spaces and Equivariant Maps}
\label{sec: Model_spaces_classification}
Before turning to the classification of model spaces of step and rank three, we recall the classification of step two model spaces given in \cite{sub_Riemannian_Model_Spaces} for comparison: The only Carnot model space with step two and rank $n$ is the free nilpotent Lie group $\mathsf{F}[n,2]$. Those model spaces of step two that are not Carnot groups are the universal covering spaces of the isometry groups of the non-flat Riemannian model spaces with suitable structures. Hence model spaces with step two have free nilpotentization and compatible Lie group structures. Moreover, they are parametrized by a single parameter, namely the sectional curvature of the corresponding non-flat Riemannian model space. 

\subsection{Carnot Model Spaces}
\label{sec: Carnot_Model_Spaces}
We will now classify all Carnot model spaces with step and rank three. The result will be used as an invariant for general model spaces through Proposition \ref{Nilpotentization_Model_Space}. In what follows, if $q \mapsto \rho(q)$ is a representation of $\Ort(3)$ on the vector space $V$, we will write $\bar{V}$ for the representation $q \mapsto (\det q) \rho(q)$. Throughout the paper, if $V_1$ and $V_2$ are representations of a Lie group $K$, we use the notation $V_1 \simeq V_2$ to indicate that $V_1$ and $V_2$ are isomorphic as representations of $K$. 
Recall that \[\mathfrak{f} := \mathsf{f}[3,3] = \mathfrak{f}_1 \oplus \mathfrak{f}_2 \oplus \mathfrak{f}_3\] denotes the free nilpotent Lie algebra of step and rank three. It follows from Witt's formula~\eqref{Witts_Formula} that $\textrm{dim}(\mathfrak{f}_1) = \textrm{dim}(\mathfrak{f}_2) = 3$ and $\textrm{dim}(\mathfrak{f}_3) = 8$. \par We consider the representation of $\Ort(\mathfrak{f}_1)$ on $\mathfrak{f}$ induced by the standard representation of $\Ort(\mathfrak{f}_1)$ on $\mathfrak{f}_1$. Fixing an orthonormal basis $A_1,A_2,A_3$ for $\mathfrak{f}_1$ gives the basis $A_{12}, A_{23}, A_{31}$ for $\mathfrak{f}_2$ defined by $A_{ij} := [A_i, A_j]$. This identifies $O(\mathfrak{f}_1)$ with $\Ort(3)$ as Lie groups and hence also $\mathfrak{f}_1$ with $\mathbb{R}^3$ as representations of $\Ort(3)$. Moreover, we have the identifications \[\mathfrak{f}_2 \simeq \wedge^2 \mathbb{R}^3 \simeq \mathfrak{so}(3) \simeq \bar{\mathbb{R}}^3,\] where $\mathfrak{so}(3)$ denotes the Lie algebra of $\Ort(3)$. The isomorphism between $\mathfrak{f}_2$ and $\wedge^2 \mathbb{R}^3$ is given by the map $[A,B] \mapsto A \wedge B$ while the isomorphism between $\wedge^2 \mathbb{R}^3$ and $\mathfrak{so}(3)$ as representations of $\Ort(3)$ is well-known. Finally, the isomorphism between $\mathfrak{so}(3)$ and $\bar{\mathbb{R}}^3$ is given by the map $\star: \mathfrak{so}(3) \to \mathbb{R}^3$ that is defined by the identity \[Ax = (\star A) \times x, \quad A \in \mathfrak{so}(3), \, x \in \mathbb{R}^3.\] If $e_1$, $e_2$, $e_3$ denote the standard basis elements in $\mathbb{R}^3$ and we define $A_{ij} := e_j e_i^t - e_i e_j^t$ for $ i \neq j$, then $\star(A_{12}) = e_3,$ $\star(A_{23}) = e_1,$ and $\star(A_{31}) = e_2.$

Consider the morphism of representations 
\[
\mathfrak{f}_1 \otimes \mathfrak{f}_2 \simeq \mathbb{R}^3 \otimes \bar{\mathbb{R}}^3 \longrightarrow \mathfrak{f}_3, \qquad A \otimes B \longmapsto [A,B].\]
By the Jacobi identity of the Lie bracket, the kernel of the above map is spanned by the element \[A_1 \otimes A_{23} + A_2 \otimes A_{31} + A_3 \otimes A_{12}.\] Notice that $\Ort(3)$ acts on the traceless $3 \times 3$ matrices $\mathfrak{sl}(3,\mathbb{R})$ by conjugation. By using the notation $A_{i,jk} := [A_i, [A_j, A_k]]$ we can identify $\mathfrak{f}_3$ with the representation $\overline{\mathfrak{sl}(3,\mathbb{R})}$ through the map
$$\sum_{j=1}^3\left(a_{j,12} A_{j,12} + a_{j,23} A_{j,23} + a_{j,31} A_{j,31}\right) \longmapsto \left(\begin{array}{ccc}
a_{1,23} & a_{1,31} & a_{1,12} \\
a_{2,23} & a_{2,31} & a_{2,12} \\
a_{3,23} & a_{3,31} & a_{3,12}
 \end{array} \right).$$
It is straightforward to verify that the only sub-representations of $\mathfrak{sl}(3,\mathbb{R})$ are $\mathfrak{so}(3)$ and the space of all symmetric matrices with zero trace $\mathfrak{s}$.

In summary, we can describe $\mathfrak{f} = \mathsf{f}[3,3]$ as follows. Define the operation
\begin{equation}
\label{symmetrization_product}
    \odot:\mathbb{R}^3 \times \mathbb{R}^3 \longrightarrow \mathfrak{s}, \qquad (x,y) \longmapsto x \odot y := \frac{1}{2} (y x^t + x y^t) - \frac{\langle x, y \rangle}{3} I_{3 \times 3}.
\end{equation}
Then we can identify $\mathfrak{f}$ as the vector space $\mathbb{R}^3 \oplus \mathbb{R}^3 \oplus \mathfrak{s} \oplus \mathbb{R}^3$ with brackets,
\begin{align}
\label{nilpotent_bracket_free}
    \left [ \begin{pmatrix}
x_1 \\ y_1 \\ S_1 \\ z_1 
\end{pmatrix}, \begin{pmatrix}
x_2 \\ y_2 \\ S_2 \\ z_2 
\end{pmatrix} \right ] = \begin{pmatrix}
0 \\ x_1 \times x_2 \\ x_1 \odot y_2 - y_1 \odot x_2 \\ x_1 \times y_2 + y_1 \times x_2
\end{pmatrix},
\end{align}
for $x_j, y_j, z_j \in \mathbb{R}^3$, $S_j \in \mathfrak{s}$, $j=1,2.$
As a representation of $\Ort(3)$, $\mathfrak{f}$ is isomorphic to $\mathbb{R}^3 \oplus \bar{\mathbb{R}}^3 \oplus \bar{\mathfrak{s}} \oplus \mathbb{R}^3$ where the action is given by
$$
q \cdot \begin{pmatrix} x \\ y \\ S \\ z \end{pmatrix}  := d_{1}\varphi_q \begin{pmatrix} x \\ y \\ S \\ z \end{pmatrix} 
=
\begin{pmatrix} q x \\ (\det q) q y \\ (\det q) q Sq^{-1} \\ q z \end{pmatrix}, \qquad q \in \Ort(3), \, S \in \mathfrak{s}, \, x,y,z \in \mathbb{R}^3.$$
The sub-representations of $\mathfrak{f}$ that are also ideals and do not contain the center $\mathfrak{f}_3$ are given by
$$\mathfrak{a} = \left\{ (0, 0, 0, z)^t : z \in \mathbb{R}^3 \right\}, \qquad \mathfrak{b} = \left\{ (0, 0, S, 0)^t : S \in \mathfrak{s} \right\}.$$ Combining this with Proposition \ref{strategy_carnot_groups} gives the following result. 
\begin{theorem}
\label{Classification_Carnot_Spaces}
Let $(N, E, h)$ be a Carnot model space of step and rank three. Then the Lie algebra of $N$ is isomorphic to $\mathsf{f}[3,3]$, $\mathsf{f}[3,3]/\mathfrak{a}$, or $\mathsf{f}[3,3]/\mathfrak{b}$.
\end{theorem}
We will write $\mathsf{A}_{3,3}$ and $\mathsf{C}_{3,3}$ for the model spaces corresponding to respectively $\mathsf{f}[3,3]/\mathfrak{a}$ and $\mathsf{f}[3,3]/\mathfrak{b}$. We obtain from Proposition \ref{Nilpotentization_Model_Space} the following numerical consequence of Theorem \ref{Classification_Carnot_Spaces}.

\begin{corollary}
Let $(M,H,g)$ be a sub-Riemannian model space of step and rank three. Then the growth vector of $M$ is either $(3,6,9)$, $(3,6,11)$, or $(3,6,14)$.
\end{corollary}

For the rest of the paper, we will use the following notation: Let $e_1, e_2, e_3$ denote the standard basis of $\mathbb{R}^3$. We introduce respectively symmetric and anti-symmetric matrices
\begin{equation}
\label{skew_symmetric matricies}
    S_{ij} = e_j e_i^t + e_i e_j^t, \qquad A_{ij} = e_j e_i^t - e_i e_j^t, \qquad i\neq j,
\end{equation}
and for $j =1, 2,3$, we define $D_j = e_j e_j^t$. The bracket relations of these matrices when $i$,$j$,$k$ are all different are
$$[A_{ij}, A_{jk}]  = -A_{ik}, \quad [A_{ij}, S_{jk}] =  - S_{ik} , \quad  [S_{ij}, S_{jk}] =  - A_{ik}, \quad [A_{ij} , S_{ij} ] = 2D_j - 2D_i,$$
and for any $s$,
$$[A_{ij}, D_{s}] =( \delta_{si} - \delta_{si}) S_{ij}, \qquad [S_{ij}, D_s ] = \delta_{si} A_{sj} + \delta_{sj} A_{si}, \quad [D_i, D_s] =0.$$

\subsection{Equivariant Maps Between Representations of $\Ort(3)$} \label{sec:Equivariant}
We will now provide a technical result regarding representation theory of $\Ort(3)$ that will be used in the proofs of Lemma \ref{A_33_Classification} and Theorem \ref{nilpotent_classification}. It concerns determining the possible equivariant bilinear maps between the representations on $\mathbb{R}^3$, $\bar{\mathbb{R}}^3$, and $\mathfrak{s}$ described in the previous section. This is equivalent to understanding invariant linear maps from their tensor products, and we have the following result.

\begin{lemma}
\label{All_The_Invariant_Maps}
There is a unique (up to scaling) non-zero $\Ort(3)$-equivariant map between the following representations:
\begin{align*}
\mathbb{R}^3 \otimes \mathbb{R}^3 & \longrightarrow \bar{\mathbb{R}}^3 \tag{M1}\label{M1}, & \quad (v,w) & \longmapsto v \times w, \\
\mathbb{R}^3 \otimes \mathbb{R}^3 & \longrightarrow \mathfrak{s} \tag{M2}\label{M2}, & \quad (v,w) & \longmapsto v \odot w, \\
\mathbb{R}^3 \otimes \mathfrak{s} & \longrightarrow \mathbb{R}^3 \tag{M3}\label{M3}, & \quad (v,A) & \longmapsto Av, \\
\mathbb{R}^3 \otimes \mathfrak{s} & \longrightarrow \bar{\mathfrak{s}} \tag{M4}\label{M4}, & \quad (v,A) & \longmapsto [\star^{-1} v,A], \\
\mathfrak{s} \otimes \mathfrak{s} & \longrightarrow \bar{\mathbb{R}}^3 \tag{M5}\label{M5}, & \quad (A,B) & \longmapsto \star [A,B].
\end{align*}
\end{lemma}
Notice that $V_1 \otimes V_2 \simeq \bar{V}_1 \otimes \bar{V}_2$ and $\bar{V}_1 \otimes V_2 \simeq V_1 \otimes \bar{V}_2 \simeq \overline{V_1 \otimes V_2}$ hold for any representations $V_1$ and $V_2$ of $\Ort(3)$. Determining all equivariant maps from $V_1 \otimes V_2$ to a representation $W$ thus also determines all equivariant maps from $\bar{V}_1 \otimes V_2$ to~$\bar{W}$.
\begin{proof}
The cases \eqref{M1} and \eqref{M2} are determined by the decomposition \[\mathbb{R}^3 \otimes \mathbb{R}^3 \simeq \mathfrak{s} \oplus \mathfrak{so}(3) \oplus \mathbb{R} I_{3 \times 3},\] and by recalling that the representations of $\Ort(3)$ on $\mathfrak{so}(3) \simeq \bar{\mathbb{R}}^3$ and $\mathfrak{s}$ are both irreducible. For the case \eqref{M3}, we note that since we are considering representations of $\Ort(3)$, such maps are equivalent to considering equivariant maps $\mathfrak{s} \to \mathbb{R}^3 \otimes \mathbb{R}^3$. The only non-zero maps in the latter mentioned class are constant multiples of the inclusion.

For the remaining two cases, we will only prove \eqref{M5} as the proof of \eqref{M4} is similar. Let us consider an equivariant map $L : \mathfrak{s} \otimes \mathfrak{s} \to \bar{\mathbb{R}}^3$ that is $O(3)$-equivariant and hence $\mathfrak{so}(3)$-equivariant. This implies that
$$v \times L(S, T) = L([\star^{-1} v, S], T) + L(S, [\star^{-1} v, T]), \quad v \in \mathbb{R}^3, \, S,T \in \mathfrak{s}.$$
We look for the eigenvalue decomposition of maps on the form $x \mapsto v \times x$ for $v \in \mathbb{R}^3$. Complexity the representations and extend $L$ by linearity, so that we have well defined eigenspaces. The map $x\mapsto e_1\times x$ decomposes $\mathbb{C}^3$ into respectively a $0$, $-i$, and $i$ eigenspace given by
$$\mathbb{C}^3 = \mathfrak{a}_0 \oplus \mathfrak{a}_{i} \oplus \mathfrak{a}_{-i} = \spn_{\mathbb{C}} \{ e_1 \} \oplus \spn_{\mathbb{C}} \{ e_2 - i e_3\} \oplus \spn_{\mathbb{C}} \{ e_2 + i e_3 \}.$$
We have a similar decomposition of $\mathfrak{s} \otimes \mathfrak{s}$ up to the kernel of $L$. Let $E(\lambda)$ denote the eigenspace corresponding to $\lambda \in \mathbb{C}$ for the action of $\star^{-1} e_1$ on $\mathfrak{s}$. Then we obtain, 
\begin{align*}
D_2 - D_3 + i S_{23} \in & E(-2i), \quad S_{12}+ i S_{13} \in E(-i), \quad 2 D_1 - D_2 - D_3 \in E(0), \\ & S_{12}- i S_{13} \in E(i), \quad D_2 - D_3 - i S_{23} \in E(2i).
\end{align*}

Hence, for the action of $\star^{-1} e_1$ on $\mathfrak{s} \otimes \mathfrak{s}$, if $S_{a_j} \in \mathfrak{s}$ have eigenvalue $a_j$, $j=1,2$, then $S_{a_1} \otimes S_{a_2}$ has eigenvalue $a_1 +a_2$. We can collect the image of $L$ on the different eigenspaces of $\mathfrak{s} \otimes \mathfrak{s}$ in the following table
$$ \tiny
\begin{array}{l|c|c|c|c|c|c|c|c|c|}
L & D_2 - D_3 + i S_{23} & S_{12}+ i S_{13} & 2 D_1 - D_2 - D_3 & S_{12}- i S_{13} & D_2 - D_3 - i S_{23} \\ \hline
D_2 - D_3 + i S_{23}  & 0 & 0 & 0 & \mathfrak{a}_{-1} & \mathfrak{a}_0 \\ 
S_{12}+ i S_{13} & 0 & 0 & \mathfrak{a}_{-1} & \mathfrak{a}_0 & \mathfrak{a}_{1}  \\
2 D_1 - D_2 - D_3 & 0 & \mathfrak{a}_{-1} & \mathfrak{a}_0 & \mathfrak{a}_{1} & 0 \\
S_{12}- i S_{13} &  \mathfrak{a}_{-1} & \mathfrak{a}_0 & \mathfrak{a}_{1} & 0 & 0 \\
D_2 - D_3 - i S_{23} &  \mathfrak{a}_0 & \mathfrak{a}_{1} & 0 & 0 & 0
\end{array}$$
Similarly, we can consider the action of $\star^{-1} e_2$ and $\star^{-1} e_3$ with respective eigenvalue decomposition $\mathfrak{b}_0 \oplus \mathfrak{b}_i \oplus \mathfrak{b}_{-i}$ and $\mathfrak{c}_0 \oplus \mathfrak{c}_i \oplus \mathfrak{c}_{-i}$ and obtain
$$\tiny
\begin{array}{l|c|c|c|c|c|c|c|c|c|}
L & D_3 - D_1 + i S_{13} & S_{23}+ i S_{12} & 2 D_2 - D_1 - D_3 & S_{23}- i S_{12} & D_3 - D_1 - i S_{13} \\ \hline
D_3 - D_1 + i S_{13}  & 0 & 0 & 0 & \mathfrak{b}_{-1} & \mathfrak{b}_0 \\ 
S_{23}+ i S_{12} & 0 & 0 & \mathfrak{b}_{-1} & \mathfrak{b}_0 & \mathfrak{b}_{1}  \\
2 D_2 - D_1 - D_3 & 0 & \mathfrak{b}_{-1} & \mathfrak{b}_0 & \mathfrak{b}_{1} & 0 \\
S_{23}- i S_{12} &  \mathfrak{b}_{-1} & \mathfrak{b}_0 & \mathfrak{b}_{1} & 0 & 0 \\
D_3 - D_1 - i S_{13} &  \mathfrak{b}_0 & \mathfrak{b}_{1} & 0 & 0 & 0
\end{array}$$
$$ \tiny
\begin{array}{l|c|c|c|c|c|c|c|c|c|}
L & D_1 - D_2 + i S_{12} & S_{13}+ i S_{23} & 2 D_3 - D_1 - D_2 & S_{13}- i S_{23} & D_1 - D_2 - i S_{12} \\ \hline
D_1 - D_2 + i S_{12}  & 0 & 0 & 0 & \mathfrak{c}_{-1} & \mathfrak{c}_0 \\ 
S_{13}+ i S_{23} & 0 & 0 & \mathfrak{c}_{-1} & \mathfrak{c}_0 & \mathfrak{c}_{1}  \\
2 D_3 - D_1 - D_2 & 0 & \mathfrak{c}_{-1} & \mathfrak{c}_0 & \mathfrak{c}_{1} & 0 \\
S_{13}- i S_{23} &  \mathfrak{c}_{-1} & \mathfrak{c}_0 & \mathfrak{c}_{1} & 0 & 0 \\
D_1 - D_2 - i S_{12} &  \mathfrak{c}_0 & \mathfrak{c}_{1} & 0 & 0 & 0
\end{array}$$

In summary, if we write $\mathfrak{e}_j = \spn \{ e_j\}$ for $j=1,2,3$ we have
$$
\begin{array}{l|c|c|c|c|c|c|c|c|c|}
L & S_{12} & S_{23}  & S_{13} & D_1-D_2 & D_2 - D_3 \\ \hline
S_{12} & 0 & \mathfrak{e}_2 & \mathfrak{e}_1 & \mathfrak{e}_3 & \mathfrak{e}_3 \\ 
S_{23} & \mathfrak{e}_2 & 0 & \mathfrak{e}_3 & \mathfrak{e}_1 &  \mathfrak{e}_1 \\
S_{13} & \mathfrak{e}_1 & \mathfrak{e}_3 & 0 & \mathfrak{e}_2 & \mathfrak{e}_2 \\
D_1 - D_2 & \mathfrak{e}_3 & \mathfrak{e}_1 & \mathfrak{e}_2 & 0 &  0 \\
D_2 -D_3 & \mathfrak{e}_3 & \mathfrak{e}_1 & \mathfrak{e}_2 & 0 & 0
\end{array}$$
Defining $c \in \mathbb{R}$ by $L(S_{12}, S_{23}) = c e_2$ and using the action of $\mathfrak{so}(3)$, we end up with
$$
\begin{array}{l|c|c|c|c|c|c|c|c|c|}
L & S_{12} & S_{23}  & S_{13} & D_1-D_2 & D_2 - D_3 \\ \hline
S_{12} & 0 & ce_2 & -c e_1 &  2c e_3 & - c e_3 \\ 
S_{23} & - ce_2 & 0 & c e_3 & - c e_1 & 2 c e_1 \\
S_{13} & c e_1 & - c e_3 & 0 & - c e_2 & - c e_2 \\
D_1 - D_2 & -2 ce_3 & c e_1 & c e_2 & 0 &  0 \\
D_2 -D_3 & c e_3 & - 2 c e_1 & c e_2 & 0 & 0
\end{array}$$
A straightforward calculation shows that this is precisely the map \[\mathfrak{s} \otimes \mathfrak{s} \ni (A,B) \longmapsto c \star[A,B] \in \bar{\mathbb{R}}^3.\]
\end{proof}

\section{Model spaces with nilpotentization $\mathsf{C}_{3,3}$}
\label{Sec: C33}
In this section we will focus on describing the homogeneous structure of the sub-Riemannian model spaces with nilpotentization $\mathsf{C}_{3,3}$. These spaces have been studied in \cite[Theorem 6.1]{sub_Riemannian_Model_Spaces}, where the following result was proven. 

\begin{theorem}
\label{C_33_Classification}
Let $(M,H,g)$ be a sub-Riemannian model space whose nilpotentization is isometric to $\mathsf{C}_{3,3}$ and fix $x \in M$. Then $M$ is isometric to $\mathsf{C}_{3,3}(a_1,a_2)$ for a unique element $(a_1,a_2) \in \mathbb{R}^2$, where $\mathsf{C}_{3,3}(a_1,a_2)$ is a model space with the following description: The Lie algebra $\mathfrak{g}$ of $G = \Isom\left(\mathsf{C}_{3,3}(a_1,a_2)\right)$ has the identification as a representation
\[\mathfrak{g} \simeq \mathbb{R}^3 \oplus \bar{\mathbb{R}}^3 \oplus \mathbb{R}^3 \oplus \bar{\mathbb{R}}^3,\]
where the last $\mathbb{R}^3$-term is identified with the Lie algebra of the isotropy group at $x$. The Lie bracket between elements in $\mathfrak{g}$ is given by
\begin{align} \label{BracketsC33} \left[\begin{pmatrix} x_1 \\ y_1 \\ z_1 \\ w_1 \end{pmatrix},\begin{pmatrix} x_2 \\ y_2 \\ z_2 \\ w_2 \end{pmatrix}\right]
& = \begin{pmatrix}
w_1 \times x_2 + x_1 \times w_2 \\
w_1 \times y_2 + y_1 \times w_2   \\
w_1 \times z_2 + z_1 \times w_2 \\
w_1 \times w_2 \end{pmatrix}
+\begin{pmatrix}
0 \\ x_1 \times x_2  \\ y_1 \times x_2 + x_1 \times y_2 \\ 0 \end{pmatrix}
\\ \nonumber & \qquad + a_1 \begin{pmatrix}
0 \\ x_1 \times z_2 + z_1 \times x_2 + y_1 \times y_2 + a_1 z_1 \times z_2 \\
y_1 \times z_2 + z_1 \times y_2 \\ 0 \end{pmatrix}
\\ \nonumber
& \qquad + a_2 \begin{pmatrix}y_1 \times z_2 + z_1 \times y_2 \\ z_1 \times z_2  \\
0 \\ x_1 \times z_2 + z_1 \times x_2 + y_1 \times y_2 + a_1 z_1 \times z_2 \end{pmatrix}.
\end{align}
The spaces $\mathsf{C}_{3,3}(a_1,a_2)$ for $(a_1,a_2) \in \mathbb{R}^2$ form a non-isometric family of sub-Riemannian model spaces, implying that $(M,H,g)$ is uniquely determined by the numbers $a_1$ and $a_2$. Finally, for any positive constant $\lambda > 0$, we note that \[\lambda \cdot \mathsf{C}_{3,3}(a_1,a_2) = \mathsf{C}_{3,3}\left(a_1/\lambda^2, a_2/\lambda^4\right).\]
\end{theorem}

Although Theorem \ref{C_33_Classification} classifies all sub-Riemannian model spaces whose nilpotentization is isometric to $\mathsf{C}_{3,3}$, the description is cumbersome and hides the homogeneous structure. In the following theorem we remedy this and provide explicit descriptions in terms of familiar spaces.
\begin{theorem} \label{th:AllC33} 
\begin{enumerate}[\rm (a)]
\item Assume either that $a_2 \geq 0$ or that both $a_2 < 0$ and $2 \sqrt{|a_2|} \leq a_1$. Then $M = \mathsf{C}_{3,3}(a_1,a_2)$ is a connected and simply connected Lie group with Lie algebra $\mathfrak{m}$ and with a sub-Riemannian structure $(H,g)$ given by left translation of $\mathfrak{h} = \{ A_x\}_{x \in \mathbb{R}^3}$ with inner product $\langle A_x , A_y \rangle = \langle x, y \rangle$, where $\mathfrak{m}$ and $A_x$ are found in Table~1. Furthermore, in all of these cases, all orientation preserving isometries are a composition of a left translation and a Lie group isomorphism.
\item Assume that both $a_2 < 0$ and $a_1 < 2 \sqrt{|a_2|}$. Then $M  = \mathsf{C}_{3,3}(a_1,a_2) = G/K$, where $G$ is the isometry group of $M$ with Lie algebra $\mathfrak{g}$ and $K$ is the isotropy group of a point $x \in M$ with Lie algebra $\mathfrak{k} = \{ C_w \}_{w \in \mathbb{R}^3}$. Furthermore, the sub-Riemannian structure $(H , g)$ on $M$ is the projection of a sub-Riemannian structure on $G$ defined by left translation of $\mathfrak{p} = \{ B_x\}_{x \in \mathbb{R}^3}$ with inner product $\langle B_x, B_y \rangle = \langle x,y \rangle$. Here, $\mathfrak{g}$ and the elements $B_x$ and $C_w$ are as in Table~2. \end{enumerate}
\end{theorem}
We will use the rest of this section to prove this result.

\begin{center}
\begin{table}[H]
\centering
\begin{tabular}{|l||l|c|  } \hline
Cases & $\mathfrak{m}$  & Basis $A_x$, $x \in \mathbb{R}^3$  \\ \hline
$a_2 >0$ & $\mathfrak{so}(3) \oplus \mathfrak{so}(3,1)$ & $A_x = \begin{pmatrix} 2\kappa \star^{-1} x \end{pmatrix} \oplus \begin{pmatrix} \kappa \star^{-1} x & \hat \kappa x \\ \hat \kappa x^t & 0 \end{pmatrix}$  \\ \hline
$a_1 < 0$, $a_2 =0$  & $\mathbb{R}^3 \oplus \mathfrak{so}(3,1)$ & $A_x = \begin{pmatrix}  x \end{pmatrix} \oplus \begin{pmatrix} 0 & \hat \kappa x  \\ \hat \kappa x^t  & 0  \end{pmatrix}$  \\ \hline
$a_1 >0$, $a_2 =0$ & $\mathbb{R}^3 \oplus \mathfrak{so}(4)$  & $A_x = \begin{pmatrix} x \end{pmatrix} \oplus \begin{pmatrix} 0 & \kappa x \\ - \kappa x^t & 0 \end{pmatrix}$ \\ \hline
$a_2 <0$, $2 \sqrt{|a_2|} < a_1$ & $\mathfrak{so}(3) \oplus \mathfrak{so}(4)$ & $A_x = \begin{pmatrix} 2\kappa \star^{-1} x \end{pmatrix} \oplus \begin{pmatrix} \kappa \star^{-1} x & \hat \kappa x \\ - \hat \kappa x^t & 0 \end{pmatrix}$  \\ \hline
$a_2 <0$, $2 \sqrt{|a_2|} = a_1$ & $\so(3) \ltimes_\theta (\mathbb{R}^3 \times \mathbb{R}^3)$ & $A_x = (2 \kappa \star^{-1} x) \ltimes \begin{pmatrix} - \frac{\kappa}{2} x \\ 0 \end{pmatrix}$ \\
& & $\theta(\star^{-1} x)\begin{pmatrix} z \\ y \end{pmatrix} = \frac{1}{2} \begin{pmatrix}  x \times (z-y) \\ x \times (y-z)\end{pmatrix} $ \\ \hline
\end{tabular}
\begin{tabular}{ll} \\
$\kappa = \frac{1}{\sqrt{2}}  \left|a_1 + \sqrt{a_1^2 + 4a_2}\right|^{1/2},$ & $\hat \kappa = \frac{1}{\sqrt{2}}  \left|a_1 - \sqrt{a_1^2 + 4a_2}\right|^{1/2}.$
\end{tabular}
\caption{Model spaces with nilpotentization $\mathsf{C}_{3,3}$ that have a sub-Riemannian structure invariant under orientation preserving isometries.
}
\end{table}
\end{center}

\begin{center}
\begin{table}[H]
\centering
\begin{tabular}{|l||l|c|  } \hline
Cases & $\mathfrak{g}$  & Basis $B_x$ , $x \in \mathbb{R}^3$, Basis $C_w$, $w \in \mathbb{R}^3$  \\ \hline
$a_2 <0$ & $\mathfrak{so}(3,1) \oplus \mathfrak{so}(3,1)$ & {\tiny $B_x = \begin{pmatrix} 0 & \kappa x \\ \kappa x^t & 0  \end{pmatrix} \oplus \begin{pmatrix} 0 & \hat \kappa x \\  \hat \kappa x^t & 0  \end{pmatrix}$} \\
$2\sqrt{|a_2|} < -a_1$& & {\tiny $C_w = \begin{pmatrix} \star^{-1} w & 0 \\ 0 & 0  \end{pmatrix} \oplus \begin{pmatrix} \star^{-1} w & 0 \\ 0 & 0  \end{pmatrix}$}  \\ & & \\
& & {\tiny $\kappa = \frac{1}{\sqrt{2}}  \left|a_1 + \sqrt{a_1^2 + 4a_2}\right|^{1/2},$ $\hat \kappa = \frac{1}{\sqrt{2}}  \left|a_1 - \sqrt{a_1^2 + 4a_2}\right|^{1/2}$}\\
\hline
$a_2 <0$ & $\so(3,1) \ltimes_\theta (\mathbb{R}^3 \times \mathbb{R}^3)$ & {\tiny $B_x = \begin{pmatrix} 0 & \sqrt{|a_1|/2} x^t \\ \sqrt{|a_1|/2} x^t & 0 \end{pmatrix} \ltimes \begin{pmatrix} -\frac{\sqrt{|a_1|/2}}{2} x \\ 0 \end{pmatrix}$}  \\ 
$2\sqrt{|a_2|} = -a_1$ & & {\tiny $C_w = \begin{pmatrix} \star^{-1} w & 0 \\ 0 & 0 \end{pmatrix} \ltimes \begin{pmatrix} 0 \\ 0 \end{pmatrix}$} \\ & &{\tiny $\theta\begin{pmatrix} \star^{-1} w & x \\ x^t & 0 \end{pmatrix}\begin{pmatrix} z \\ y \end{pmatrix} = \begin{pmatrix}  w \times z +  x \times y \\ w \times y - x \times z \end{pmatrix} $}  \\ \hline
$a_2 <0$ & $\mathfrak{so}(3,1) \oplus \mathfrak{so}(3,1)$ & {\tiny $B_x = \begin{pmatrix} \mathrm{Re}(\zeta) \star^{-1} x & \mathrm{Im}(\zeta) x \\ \mathrm{Im}(\zeta) x^t  & 0  \end{pmatrix} \oplus \begin{pmatrix} -\mathrm{Re}(\zeta) \star^{-1} x  & - \mathrm{Im}(\zeta) x  \\ -\mathrm{Im}(\zeta) x^t  & 0  \end{pmatrix}$} \\
$|a_1| <2\sqrt{|a_2|}$ & & {\tiny $C_w = \begin{pmatrix} \star^{-1} w & 0 \\ 0 & 0  \end{pmatrix} \oplus \begin{pmatrix} \star^{-1} w & 0 \\ 0 & 0  \end{pmatrix}$} \\ & & \\
& & {\tiny $\zeta =\frac{1}{\sqrt{2}}  \sqrt{a_1 + \sqrt{a_1^2 + 4a_2} }$ (principal root) }  \\ \hline
\end{tabular}
\caption{Other model spaces with nilpotentization $\mathsf{C}_{3,3}$.}
\end{table}
\end{center}
We can conclude the following from the above table.
\begin{corollary}
The model space $\mathsf{C}_{3,3}(a_1,a_2)$ is compact if and only if $a_2 < 0$ and $a_1 < 2 \sqrt{-a_2}$.
\end{corollary}

\subsection{Isometry-Invariant Lie group Structure}
As a first step towards proving Theorem \ref{th:AllC33}, we want to understand which of the sub-Riemannian model spaces $\mathsf{C}_{3,3}(a_1,a_2)$ can be made into Lie groups with left invariant structures. To do this, we will use the theory developed in Section~\ref{sec:CrossProduct}.
\begin{lemma} \label{lemma:HolCZero}
Consider the sub-Riemannian model space $\mathsf{C}_{3,3}(a_1, a_2)$ and define $\nabla^c$ as in \eqref{nablac}.
\begin{enumerate}[\rm (a)]
\item If $a_2 >0$, then $\Hol^{\nabla^c} = 1$ if and only if
$$c = \pm \frac{1}{\sqrt{2}} \sqrt{\sqrt{a_1^2 + 4a_2} + a_1}.$$
\item If $a_2 \leq 0$ and $a_1 \geq \sqrt{-4a_2}$, then $\Hol^{\nabla^c} = 1$ if and only if
$$c = \pm \frac{1}{\sqrt{2}} \sqrt{a_1 - \sqrt{a_1^2 + 4a_2}} , \qquad c = \pm \frac{1}{\sqrt{2}} \sqrt{a_1 + \sqrt{a_1^2 + 4a_2}} .$$
\item If $a_2 \leq 0$ and $a_1 < \sqrt{-4a_2}$, then $\Hol^{\nabla^c}$ is non-trivial for all values of $c$.
\end{enumerate}
\end{lemma}
\begin{proof} Recall the brackets in \eqref{BracketsC33} for the Lie algebra $\mathfrak{g}$ of the isometry group. Define $\mathfrak{p}_c$ as the span of elements on the form
$$\begin{pmatrix} x \\ 0 \\ 0  \end{pmatrix}' : = \begin{pmatrix} x \\ 0 \\ 0 \\ c x \end{pmatrix}.$$
This is the subspace of $\mathfrak{g}$ corresponding to $\nabla^c$. We want to try to find a value of $c$ such that $\mathfrak{p}_c$ generate a subalgebra transverse to $\mathfrak{k}$, meaning that $\nabla^c$ has trivial holonomy. We observe that
\begin{align*} &  \left[ \begin{pmatrix} x_1 \\ 0 \\ 0 \end{pmatrix} ',\begin{pmatrix} x_2 \\ 0 \\ 0 \end{pmatrix}' \right]
=  \begin{pmatrix} 2c x_1 \times x_2 \\ x_1 \times x_2 \\ 0 \\ c^2 x_1 \times x_2 \end{pmatrix} = \begin{pmatrix} 2c x_1 \times x_2 \\ x_1 \times x_2 \\ 0 \end{pmatrix}',  \text{ with } \begin{pmatrix} x \\ y \\ 0 \end{pmatrix}' := \begin{pmatrix} x \\ y \\ 0 \\ cx - c^2 y \end{pmatrix}.
\end{align*}
We then have
\begin{align*} &  \left[ \begin{pmatrix} x \\ 0 \\ 0 \end{pmatrix} ',\begin{pmatrix} 0 \\ y \\ 0 \end{pmatrix}' \right]
=  \begin{pmatrix} - c^2 x\times y \\ c x \times y \\ x \times y \\ - c^3 x \times y \end{pmatrix} = \begin{pmatrix} -c^2  x \times y \\ c x \times y \\ x \times y \end{pmatrix}',  \text{ with } \begin{pmatrix} x \\ y \\ z \end{pmatrix}' := \begin{pmatrix} x \\ y \\ z \\ cx - c^2 y + c^3 z \end{pmatrix}.
\end{align*}
Finally we obtain
\begin{align*} &  \left[ \begin{pmatrix} x \\ 0 \\ 0 \end{pmatrix} ',\begin{pmatrix} 0 \\ 0 \\ z \end{pmatrix}' \right]
=  \begin{pmatrix} c^3 x\times z \\ a_1 x \times z \\ c x \times z \\ (c^4+ a_2) x \times z \end{pmatrix} = \begin{pmatrix} c^3  x \times y \\ a_1 x \times y \\ cx \times y \end{pmatrix}'  + \begin{pmatrix} 0\\ 0 \\ 0 \\ (- c^4+ a_2+ c^2 a_1 ) x \times z \end{pmatrix}.\end{align*}
Hence, in order for $\mathfrak{p}_c$ to generate a proper subalgebra, we must have that 
\[a_1^2 + 4a_2 \geq 0 \quad \textrm{and} \quad c^2 = \frac{a_1}{2} \pm \frac{\sqrt{a_1^2 + 4 a_2}}{2} \geq 0.\]
\end{proof}

Combining Lemma \ref{lemma:HolCZero} with Lemma \ref{lemma:CrossProduct} (d) gives the following result about compatible Lie group structures. 

\begin{corollary}
The model space $\mathsf{C}_{3,3}(a_1,a_2)$ has a Lie group structure making the horizontal distribution and the metric left-invariant such that every orientation-preserving isometry is a composition of a left translation and a Lie group automorphism if and only if either $a_2 > 0$ or both $a_2 \leq 0$ and $a_1 \geq \sqrt{-4a_2}$.
\end{corollary}

\subsection{Proof of Theorem~\ref{th:AllC33}} 
We will prove several lemmas before turning to the proof of Theorem \ref{th:AllC33}.
Define $\mathfrak{b}_k$ for $k \in \mathbb{R}$ as the Lie algebra given by the vector space $\mathbb{R}^3 \oplus \mathbb{R}^3$ with brackets
$$\left[\begin{pmatrix} x_1 \\ y_1 \end{pmatrix} , \begin{pmatrix} x_2 \\ y_2 \end{pmatrix}\right] = \begin{pmatrix} x_1 \times y_2 + y_1 \times x_2 \\ k x_1 \times x_2 + y_1 \times y_2  \end{pmatrix}.$$
We note that $\mathfrak{b}_{1}$ and $\mathfrak{b}_{-1}$ are isomorphic to respectively $\so(4)$ and $\so(3,1)$.
\begin{lemma} \label{C33Lemma1}
Assume that $a_1^2 + 4 a_2 > 0$ and let $k$ and $\hat k$ be defined by
\begin{equation} \label{AtoK} k = \left(a_1 + \sqrt{a_1^2 + 4a_2}\right)/2, \qquad \hat k = \left(a_1 - \sqrt{a_1^2 + 4a_2}\right)/2.\end{equation}
Let $\mathfrak{g}$ denote the Lie algebra of the isometry group of $\mathsf{C}_{3,3}(a_1,a_2)$. Then there is a Lie algebra isomorphism $\psi: \mathfrak{g} \to \mathfrak{b}_k \oplus \mathfrak{b}_{\hat k}$ given by
$$\psi\begin{pmatrix} x \\ y \\ z \\ w \end{pmatrix} = \begin{pmatrix} x + kz \\ w + ky \end{pmatrix} \oplus \begin{pmatrix} x + \hat k z \\ w + \hat k y \end{pmatrix}.$$
\end{lemma}
\begin{proof}
We note that by definition, $k$ and $\hat k$ can be considered as the unique solution of
$$k + \hat k =a_1, \qquad - k \hat k = a_2, \qquad \hat k < k.$$
A straightforward computation shows that

\begin{align*}
    \left[\psi\begin{pmatrix}
    x_1 \\ y_1 \\ z_1 \\ w_1
    \end{pmatrix}, \psi\begin{pmatrix}
    x_2 \\ y_2 \\ z_2 \\ w_2
    \end{pmatrix}\right] & = \psi \begin{pmatrix} x_1 \times w_2 + w_1 \times x_2 \\ w_1 \times y_2 + y_1 \times w_2 + x_1 \times x_2  \\ x_1 \times y_2 + y_1 \times x_2 +z_1 \times w_2 + w_1  \\ w_1 \times w_2 \end{pmatrix} \\ & \quad + (k + \hat k) \psi \begin{pmatrix} 0 \\ x_1 \times z_2 + z_1 \times x_2 + y_1 \times y_2 - k \hat k z_1 \times z_2 \\ y_1 \times z_2 + z_1 \times y_2 \\ 0 \end{pmatrix} \\ & \quad - k \hat k \psi \begin{pmatrix} y_1 \times z_2 + z_1 \times y_2 \\ z_1 \times z_2 \\ 0 \\ x_1 \times z_2 + z_1 \times x_2 + y_1 \times y_2 + (k + \hat k) z_1 \times z_2 \end{pmatrix}.
\end{align*}
This implies that $\psi$ is a Lie algebra isomorphism by \eqref{AtoK}.
\end{proof}

\begin{lemma} \label{C33Lemma2}
Assume that $a^2_1 + 4 a_2 < 0$ and let $\mathfrak{g}$ be the Lie algebra of the isometry group of $\mathsf{C}_{3,3}(a_1, a_2)$. Let $\zeta \in \mathbb{C}$ be a complex number satisfying
$$a_2 = - |\zeta|^4, \qquad a_1 = 2 \mathrm{Re}(\zeta^2).$$
Then there is an isomorphism of Lie algebras $\psi: \mathfrak{g} \to \mathfrak{so}(3,1) \oplus \mathfrak{so}(3,1)$ given by
\begin{align*} \psi\begin{pmatrix} x\\ y \\ z \\ w \end{pmatrix} &= \begin{pmatrix} \star^{-1}( w + \mathrm{Re}(\zeta)  x + \mathrm{Re}(\zeta^2) y + \mathrm{Re}(\zeta^3) z) &  w + \mathrm{Im}(\zeta)  x + \mathrm{Im}(\zeta^2) y + \mathrm{Im}(\zeta^3) z \\ ( w + \mathrm{Re}(\zeta)  x + \mathrm{Re}(\zeta^2) y + \mathrm{Re}(\zeta^3) z)^t & 0 \end{pmatrix} \\
& \quad \oplus  \begin{pmatrix} \star^{-1}( w - \mathrm{Re}(\zeta)  x + \mathrm{Re}(\zeta^2) y - \mathrm{Re}(\zeta^3) z) &  w - \mathrm{Im}(\zeta)  x + \mathrm{Im}(\zeta^2) y - \mathrm{Im}(\zeta^3) z \\ ( w - \mathrm{Re}(\zeta)  x + \mathrm{Re}(\zeta^2) y - \mathrm{Re}(\zeta^3) z)^t & 0 \end{pmatrix}. \end{align*}
\end{lemma}

\begin{proof}
We denote the complexification of $\mathfrak{b}_1$ by $\mathfrak{b}_1^\mathbb{C}$ and define
$$\psi_2: \mathfrak{g} \to \mathfrak{b}_1^\mathbb{C}, \qquad \psi_2\begin{pmatrix} x \\ y \\ z \\ w \end{pmatrix} = \left( \begin{array}{cc} \zeta(x+ \zeta^2 z) \\ w + \zeta^3 y \end{array} \right).$$
Then we have 
\begin{align*}  & \left[ \psi_2 \begin{pmatrix} x_1 \\ y_1 \\ z_1 \\w_1 \end{pmatrix}, \psi_2 \begin{pmatrix} x_2 \\ y_2 \\ z_2\\ w_2\end{pmatrix}  \right]
= \begin{pmatrix} (w_1 + \zeta^2 y_1) \times \zeta(x_2 + \zeta^2 z_2) + \zeta (x_1 + \zeta^2 z_1) \times (w_2 + \zeta^2 y_2) \\ \zeta^2(x_1 + \zeta^2 z_1) \times (x_2 + \zeta^2 z_2) + (w_1 + \zeta^2 y_1) \times (w_2+ \zeta^2 y_2) \end{pmatrix}  \\
& = \psi_2 \begin{pmatrix} w_1 \times x_2 + x_1 \times w_2 \\ w_1 \times y_2 + y_1 \times w_2 + x_1 \times x_2 \\ y_1 \times x_2 + x_1 \times y_2  + w_1 \times z_2 + z_1 \times w_2 \\ w_1 \times w_2 \end{pmatrix}  \\
& \quad - |\zeta|^4 \psi_2 \begin{pmatrix}
y_1 \times z_2+ z_1 \times y_2  \\ z_1 \times z_2 \\ 0 \\ x_1 \times z_2 + z_1 \times x_2  + y_1 \times y_2 + 2\mathrm{Re}(\zeta^2) z_1 \times z_2 \end{pmatrix}
\\
&\quad + 2 \mathrm{Re}(\zeta^2) \psi_2 \begin{pmatrix} 0 \\ x_1 \times z_2 + z_1 \times x_2 + y_1 \times y_2 + 2 \mathrm{Re}(\zeta^2) z_1 \times z_2   \\ y_1 \times z_2+ z_1 \times y_2  \\ 0 \end{pmatrix}. \end{align*}
Hence $\psi_2$ is a Lie algebra isomorphism.
Next, we note that $\mathfrak{b}_1$ is isomorphic to $\so(3) \oplus \so(3)$ through the map \[(x,y) \longmapsto (\star^{-1} (y+x), \star^{-1} (y-x)).\] Finally, we know that $\mathfrak{so}(3)^\mathbb{C}$ is isomorphic to $\mathfrak{so}(3,1)$ though the map
\[(\star^{-1} x) \longmapsto \begin{pmatrix} \star^{-1} \mathrm{Re}(x) & \mathrm{Im}(x) \\ \mathrm{Im}(x)^t & 0\end{pmatrix}.\]
Composing these maps gives us the desired Lie algebra isomorphism.
\end{proof}

\begin{lemma} \label{lemma:exceptional}
Assume that $a^2_1 + 4a_2 =0$ and write $\kappa = \sqrt{|a_1|/2}$. If $a_1 >0$, then $\mathfrak{g}$ is isomorphic to the semi-direct product $\so(4) \ltimes_\theta (\mathbb{R}^3 \times \mathbb{R}^3)$ where $\theta: \so(4) \to \mathfrak{der}(\mathbb{R}^3 \times \mathbb{R}^3)$ is given by
\[\theta\begin{pmatrix} \star^{-1} w & x \\ - x^t & 0 \end{pmatrix}\begin{pmatrix}  z \\ y \end{pmatrix} = \begin{pmatrix}  w \times z-  x \times y \\ w \times y - x \times z \end{pmatrix}.\]
A Lie algebra isomorphism $\varphi$ between $\mathfrak{g}$ and $\so(4) \ltimes_\theta (\mathbb{R}^3 \times \mathbb{R}^3)$ is given by
\[\varphi\begin{pmatrix} x \\ y \\ z \\ w \end{pmatrix} = \begin{pmatrix} \star^{-1} (w + \kappa^2 y) &  - \kappa x - \kappa^3 z \\ \kappa x^t + \kappa^3 z^t & 0 \end{pmatrix} \ltimes \begin{pmatrix} \frac{3}{2}\kappa^3 z + \frac{\kappa}{2} x \\ \kappa^2 y \end{pmatrix}.\]
Similarly, if $a_1 <0$ then $\mathfrak{g}$ is isomorphic to $\so(3,1) \ltimes_\theta (\mathbb{R}^3 \times \mathbb{R}^3)$ where $\theta: \so(3,1) \to \mathfrak{der}(\mathbb{R}^3 \times \mathbb{R}^3)$ is given by
\[\theta\begin{pmatrix} \star^{-1} w & x \\ x^t & 0 \end{pmatrix}\begin{pmatrix}  z \\ y  \end{pmatrix} = \begin{pmatrix}  w \times z +  x \times y \\ w \times y - x \times z \end{pmatrix}. \]
A Lie algebra isomorphism $\varphi$ between $\mathfrak{g}$ and $\so(3,1) \ltimes_\theta (\mathbb{R}^3 \times \mathbb{R}^3)$ is given by
\[\varphi\begin{pmatrix} x \\ y \\ z \\ w \end{pmatrix} = \begin{pmatrix} \star^{-1} (w - \kappa^2 y) &   \kappa x - \kappa^3 z \\ \kappa x^t - \kappa^3 z^t & 0 \end{pmatrix} \ltimes \begin{pmatrix} \frac{3}{2}\kappa^3 z - \frac{\kappa}{2} x \\ \kappa^2 y \end{pmatrix}.\]

\end{lemma}
\begin{proof}
We will construct the isomorphism as the composition of three maps. Define $\kappa = \sqrt{|a_1|/2}$ such that $a_2= - \kappa^4$. Write $s = \sgn(a_1)$. By doing a scaling $(x,y,z,w)' = (\kappa^{-1} x, \kappa^{-2} y, \kappa^{-3} z, w)$, we get the following expression for the Lie brackets
\begin{align*}  \left[\begin{pmatrix} x_1 \\ y_1 \\ z_1 \\ w_1 \end{pmatrix}' ,\begin{pmatrix} x_2 \\ y_2 \\ z_2 \\ w_2 \end{pmatrix}' \right]
& = \begin{pmatrix}
w_1 \times x_2 + x_1 \times w_2 \\
w_1 \times y_2 + y_1 \times w_2   \\
w_1 \times z_2 + z_1 \times w_2 \\
w_1 \times w_2 \end{pmatrix}'
+\begin{pmatrix}
0 \\ x_1 \times x_2  \\ y_1 \times x_2 + x_1 \times y_2 \\ 0 \end{pmatrix}'
\\ \nonumber & \qquad + 2s \begin{pmatrix}
0 \\ x_1 \times z_2 + z_1 \times x_2 + y_1 \times y_2 + 2s z_1 \times z_2 \\
y_1 \times z_2 + z_1 \times y_2 \\ 0 \end{pmatrix}'
\\ \nonumber
& \qquad - \begin{pmatrix}y_1 \times z_2 + z_1 \times y_2 \\ z_1 \times z_2  \\
0 \\ x_1 \times z_2 + z_1 \times x_2 + y_1 \times y_2 + 2 s  z_1 \times z_2 \end{pmatrix}'.
\end{align*}
Next, if we write $(x,y,z,w)'' = (x- sz, y, z, w - s y)'$ we obtain 
{\tiny \begin{align*} \left[\begin{pmatrix} x_1 \\ y_1 \\ z_1 \\ w_1 \end{pmatrix}'',\begin{pmatrix} x_2 \\ y_2 \\ z_2 \\ w_2 \end{pmatrix}'' \right]
& = \begin{pmatrix}
(w_1-sy_1) \times (x_2 - sz_2) + (x_1 -s z_1) \times (w_2 - sy_2) \\
(w_1 -sy_1) \times y_2 + y_1 \times (w_2 -sy_2) + (x_1- sz_1) \times (x_2-sz_2)  \\
(w_1- sy_1) \times z_2 + z_1 \times (w_2- sy_2) + y_1 \times (x_2- s z_2) + (x_1- sz_1) \times y_2\\
(w_1- sy_1) \times (w_2- sy_2) \end{pmatrix}' \\
& \qquad + 2s \begin{pmatrix}
0 \\ (x_1-sz_1) \times z_2 + z_1 \times (x_2- s z_2) + y_1 \times y_2 + 2s z_1 \times z_2 \\
y_1 \times z_2 + z_1 \times y_2 \\ 0 \end{pmatrix}'
\\ \nonumber
& \qquad - \begin{pmatrix}y_1 \times z_2 + z_1 \times y_2 \\ z_1 \times z_2  \\
0 \\ (x_1- sz_1) \times z_2 + z_1 \times (x_2-sz_1) + y_1 \times y_2 + 2s z_1 \times z_2 \end{pmatrix}' \\
& = \begin{pmatrix}
w_1 \times x_2 -sy_1 \times x_2  +  x_1 \times w_2 - s x_1 \times y_2 -s w_1 \times z_2 - s z_1 \times w_2 \\
w_1  \times y_2 + y_1 \times w_2 + x_1 \times x_2 + sz_1 \times x_2 + s  x_1 \times z_2   \\
w_1 \times z_2 + z_1 \times w_2 + y_1 \times x_2+ x_1 \times y_2\\
w_1 \times w_2 - s y_1 \times w_2 - s w_1 \times y_2 - s^2(x_1 \times z_2 + z_1 \times x_2  ) \end{pmatrix}' \\
& = \begin{pmatrix}
w_1 \times x_2  +  x_1 \times w_2\\
w_1  \times y_2 + y_1 \times w_2 + x_1 \times x_2 + s (z_1 \times x_2 +  x_1 \times z_2)  \\
w_1 \times z_2 + z_1 \times w_2 + y_1 \times x_2+ x_1 \times y_2\\
w_1 \times w_2 + s x_1 \times x_2  \end{pmatrix}''.
\end{align*} }
Finally, if we define $(x, y, z, w)''' = (-s x, y, z + x/2, w)''$ then the Lie bracket becomes
{\tiny \begin{align*} \left[\begin{pmatrix} x_1 \\ y_1 \\ z_1 \\ w_1 \end{pmatrix}''',\begin{pmatrix} x_2 \\ y_2 \\ z_2 \\ w_2 \end{pmatrix}''' \right]
& = \begin{pmatrix}
 - s w_1 \times x_2  - s  x_1 \times w_2\\
w_1  \times y_2 + y_1 \times w_2  - s^2  (z_1 \times x_2 + x_1 \times z_2)  \\
w_1 \times (z_2+x_2/2) + (z_1 +x_1/2) \times w_2 - s y_1 \times x_2 -2s x_1 \times y_2\\
w_1 \times w_2 + s^3 x_1 \times x_2  \end{pmatrix}'' \\
& = \begin{pmatrix}
w_1 \times x_2  +  x_1 \times w_2\\
w_1  \times y_2 + y_1 \times w_2  -   (z_1 \times x_2 + x_1 \times z_2)  \\
w_1 \times z_2 + z_1 \times w_2 - s (y_1 \times x_2 + x_1 \times y_2)\\
w_1 \times w_2 + s x_1 \times x_2  \end{pmatrix}''',
\end{align*}}
which are the exactly the desired form of the brackets, showing that the Lie algebras are isomorphic. To get an explicit expression, we combine these maps and inverting them to get
\begin{align*}
& (x,y,z,w) = (\kappa x, \kappa^2 y, \kappa^3 z, w)' \\
& = (\kappa x + s \kappa^3 z, \kappa^2 y, \kappa^3 z, w + s \kappa^2 y)'' \\
& = ( - s \kappa x - \kappa^3 z, \kappa^2 y,  \frac{3}{2}\kappa^3 z + \frac{s \kappa}{2} x, w + s \kappa^2 y)'''.
\end{align*}
\end{proof}

\begin{proof}[Proof of Theorem~\ref{th:AllC33}]
\begin{enumerate}[(a)]
\item 
If $a_2 > 0$, then $\hat k = - \hat \kappa^2 < 0 < k = \kappa^2$. By Lemma~\ref{C33Lemma1}, we have the following isomorphism $\mathfrak{g} \to \mathfrak{b}_{\kappa^2} \oplus \mathfrak{b}_{- \hat \kappa^2} \to \mathfrak{b}_{1} \oplus \mathfrak{b}_{-1} \to \so(4) \otimes \so(3,1)$,
\begin{align*}
& \begin{pmatrix} x \\ y \\ z \\ w \end{pmatrix} \longmapsto \begin{pmatrix} x + \kappa^2 z \\ w + \kappa^2 y \end{pmatrix} \oplus \begin{pmatrix} x - \hat \kappa^2 z \\ w - \hat \kappa^2 y \end{pmatrix} \longmapsto \begin{pmatrix} \kappa(x + \kappa^2 z) \\ w + \kappa^2 y \end{pmatrix} \oplus \begin{pmatrix} \hat \kappa (x - \hat \kappa^2 z) \\ w - \hat \kappa^2 y \end{pmatrix} \\ &  \longmapsto \begin{pmatrix} \star^{-1} (w + \kappa^2 y) & \kappa(x + \kappa^2 z) \\ - \kappa (x + \kappa^2 z)^t & 0  \end{pmatrix} \oplus \begin{pmatrix} \star^{-1} ( w + \hat \kappa^2 y) & \hat \kappa (x - \hat \kappa^2 z) \\ \hat \kappa^2 (x - \hat \kappa^2 z)^t  & 0  \end{pmatrix}.
\end{align*}
By Lemma~\ref{lemma:HolCZero}, we can consider $\mathfrak{p}_c$ with $c = \sqrt{k} = \kappa$. Its image in $\so(4) \times \so(3,1)$ equals
$$\begin{pmatrix} \kappa \star^{-1} x & \kappa x \\ - \kappa x^t & 0  \end{pmatrix} \oplus \begin{pmatrix} \kappa \star^{-1} x & \hat \kappa x \\ \hat \kappa x^t & 0  \end{pmatrix}.$$
This subspace generate the Lie subalgebra
$$\begin{pmatrix} \star^{-1} \tilde x & \tilde x \\ - \tilde x^t & 0  \end{pmatrix} \oplus \begin{pmatrix} \star^{-1} \tilde y & \tilde z \\ \tilde z^t & 0  \end{pmatrix} , \qquad \tilde x, \tilde y, \tilde z \in \mathbb{R}^3.$$
Moreover, this Lie subalgebra is isomorphic to $\so(3) \times \so(3,1)$ through the map
$$\begin{pmatrix} \star^{-1} \tilde x & \tilde x \\ - \tilde x^t & 0  \end{pmatrix} \oplus \begin{pmatrix} \star^{-1} \tilde y & \tilde z \\ \tilde z^t & 0  \end{pmatrix} \longmapsto \begin{pmatrix} 2\star^{-1} \tilde x  \end{pmatrix} \oplus \begin{pmatrix} \star^{-1} \tilde y & \tilde z \\ \tilde z^t & 0  \end{pmatrix}.$$
The other cases are proved similarly.
\item The cases in (b) with $a_1^2 \neq 4a_2$ follow from Lemma~\ref{C33Lemma1}, Lemma~\ref{C33Lemma2}, and by similar techniques as with the cases in (a).
\end{enumerate}
\end{proof}

\subsection{Application to Other Dimensions} In general, we can define $\mathsf{C}_{n,3}$ as the connected and simply connected Lie group whose Lie algebra can be considered as $\mathbb{R}^n \times \so(n) \times \mathbb{R}^n$ with brackets,
$$[(x, A, u) , (y,B,v)] = (0, y x^t - x y^t, Ay - Bx) , \qquad x,y,u,v \in \mathbb{R}^n, \, A,B \in \so(n).$$
If we define a sub-Riemannian structure by left translation of elements $(x, 0,0)$ with inner product $\langle (x,0,0) , (y,0,0) \rangle= \langle x, y\rangle$, then $\mathsf{C}_{n,3}$ becomes a Carnot model space with growth vector \[\left(n, \frac{1}{2} n(n+1), \frac{1}{2} n(n+2)\right).\] For the case $n=2$, we have that $\mathsf{C}_{2,3} = \mathsf{F}[2,3]$ is the only Carnot model space of step three and rank two.
The results of \cite[Theorem 6.1]{sub_Riemannian_Model_Spaces} also characterizes all model spaces with this Carnot model space as their nilpotentization in terms of two parameters $(a_1, a_2) \in \mathbb{R}^2$. The proofs we have presented above that does not rely on the specific properties of dimension $n = 3$ can be converted into the case of arbitrary $n$. To be more precise, results still hold as long as they do not rely on the other partial connections $\nabla^c$ invariant under orientation preserving isometries or the Lie algebra isomorphism between $\so(4)$ and $\so(3) \oplus \so(3)$. This means that the following holds for $(M, H, g) := \mathsf{C}_{n,3}(a_1, a_2)$:
\begin{enumerate}[$\bullet$]
\item If $a_2 = 0$ and $a_1 \neq 0$, then $M$ is a Lie group with Lie algebra $\mathbb{R}^n \oplus \so(n+1)$ or $\mathbb{R}^n \oplus \so(n,1)$ when $a_1$ is respectively positive or negative. The sub-Riemannian structure $(H, g)$ is defined by left translation of elements
$$A_x = (x) \oplus \begin{pmatrix} 0 & \sqrt{|a_1|} x \\ - \mathrm{sign}(a_1) \sqrt{|a_1|} x^t & 0 \end{pmatrix},$$
with inner product $\langle A_x, A_y \rangle = \langle x, y \rangle$.
\item If $a_2 \neq 0$ and $a_1^2 + 4a_2 >0$, then $M = G/K$ where $G$ is the isometry group and $K$ is the isotropy group corresponding to a fixed point. The sub-Riemannian structure $(H,g)$ on $M$ is the projection of a left-invariant sub-Riemannian structure on $G$ defined by left translation of elements $\{ B_x \}_{x \in \mathbb{R}^n}$ with $\langle B_x, B_y \rangle = \langle x, y \rangle$ and $\mathfrak{g}$ and $B_x$ is given by Table~3.
\begin{table}[H]
\centering
\begin{tabular}{|l||l|c|  } \hline
Cases & $\mathfrak{g}$  & Basis $B_x$ , $x \in \mathbb{R}^3$, Basis $C_w$, $w \in \mathbb{R}^3$  \\ \hline
$a_2 >0$ & $\mathfrak{so}(n+1) \oplus \mathfrak{so}(n,1)$ & $B_x = \begin{pmatrix} 0 & \kappa x \\ - \kappa x^t & 0 \end{pmatrix} \oplus \begin{pmatrix} 0 & \hat \kappa x \\ \hat \kappa x^t & 0 \end{pmatrix}$  \\ & & \\ \hline
$a_2 <0$, $2 \sqrt{|a_2|} < a_1$ & $\mathfrak{so}(n+1) \oplus \mathfrak{so}(n+1)$ & $B_x = \begin{pmatrix} 0 & \kappa x \\ - \kappa x^t & 0 \end{pmatrix} \oplus \begin{pmatrix} 0 & \hat \kappa x \\ - \hat \kappa x^t & 0 \end{pmatrix}$  \\ & & \\ \hline
$a_2 <0$, $2\sqrt{|a_2|} < -a_1$ & $\mathfrak{so}(n,1) \oplus \mathfrak{so}(n,1)$ & $B_x = \begin{pmatrix} 0 & \kappa x \\ \kappa x^t & 0  \end{pmatrix} \oplus \begin{pmatrix} 0 & \hat \kappa x \\  \hat \kappa x^t & 0  \end{pmatrix}$ \\ & & \\
\hline
\end{tabular}
\caption{Here $\kappa = \frac{1}{\sqrt{2}}  \left|a_1 + \sqrt{a_1^2 + 4a_2}\right|^{1/2}$ and $\hat \kappa = \frac{1}{\sqrt{2}}  \left|a_1 - \sqrt{a_1^2 + 4a_2}\right|^{1/2}$.}
\end{table}
In all of these cases, the Lie algebra of $\mathfrak{k}$ consists of elements
$$C_A = \begin{pmatrix} A & 0 \\ 0 & 0 \end{pmatrix} \oplus \begin{pmatrix} A & 0 \\ 0 & 0 \end{pmatrix}, \qquad A \in \so(n).$$

\item If $a_1^2 + 4a_2 <0$, then $M = G/K$ where $G$ is the isometry group with Lie algebra $\so(n+1, \mathbb{C})$, $K$ is an isotropy group with Lie algebra $\mathfrak{k} = \{ C_A \, : \, A \in \so(n) \}$ given by
\[C_A = \begin{pmatrix} A & 0 \\ 0 & 0 \end{pmatrix}.\]
The sub-Riemannian structure $(H,g)$ on $M$ is the projection of a left-invariant sub-Riemannian structure on $G$ given by
$$B_x = \begin{pmatrix} 0 & \zeta x \\ - \zeta x^t  & 0 \end{pmatrix}, \qquad \langle B_x, B_y \rangle = \langle x, y \rangle,$$
with $\zeta =\frac{1}{\sqrt{2}}  \sqrt{a_1 + \sqrt{a_1^2 + 4a_2} }$ (principal root).
\item If $a^2_1 + 4 a_2 =0$, with $\kappa = \sqrt{|a_1|/2}$, then the isometry algebra is isomorphic to respectively $\so(n+1) \times_\theta (\mathbb{R}^n \times \mathbb{R}^{n(n-1)/2})$ or $\so(n,1) \times_\theta (\mathbb{R}^n \times \mathbb{R}^{n(n-1)/2})$ depending on the sign of $a_1$. If we identify $\mathbb{R}^{n(n-1)}$ with the space of anti-symmetric matrices without the Lie bracket, then
\[\theta\begin{pmatrix} A & x \\ \mp x^t & 0 \end{pmatrix}\begin{pmatrix} z \\ Y \end{pmatrix} = \begin{pmatrix}  Az \pm Yx  \\ [W , Y] - x \wedge z \end{pmatrix} .\]
The isotropy group $K$ has Lie algebra $\mathfrak{k} = \{ C_A \, : \, A \in \so(n) \}$ given by
\[C_A = \begin{pmatrix} A & 0 \\ 0 & 0 \end{pmatrix} \ltimes \begin{pmatrix} 0 \\ 0 \end{pmatrix},\]
and $(H,g)$ is the projection of a left-invariant sub-Riemannian structure on $G$ given by
\[B_x = \begin{pmatrix} 0 & \kappa x \\ \mp \kappa x^t  & 0 \end{pmatrix} \ltimes \begin{pmatrix} \mp \frac{\kappa}{2} x \\ 0 \end{pmatrix}, \qquad \langle B_x, B_y \rangle = \langle x, y \rangle.\]
\end{enumerate}

\section{Model spaces with nilpotentization $\mathsf{A}_{3,3}$}
\label{sec:(3,6,11)-classification}
For the model spaces $(M,H,g)$ with nilpotentization isometric to $\mathsf{A}_{3,3}$, we have no previous result to rely on. We will show that other than the nilpotentization, the only sub-Riemannian structures possible relate to different real forms of the exceptional (complex) Lie group $\mathrm{G}_2$.
Let $\mathfrak{g}_2$ denote the Lie algebra of $\mathrm{G}_2$ and let $\mathfrak{g}_2^s$ and $\mathfrak{g}^c_2$ denote respectively its split form and its compact form. We will present a description of these that can be found in the paper \cite{Dra18}.

We first start with the split form $\mathfrak{g}_2^s$. This can be considered as the space of matrices
$$\begin{pmatrix} 0 & -2 y^t & -2x^t \\ x & S + \star^{-1} w & \star^{-1} y \\ y & \star^{-1} x & - S + \star^{-1} w \end{pmatrix}, \qquad x,y, w \in \mathbb{R}^3, \, S \in \mathfrak{s}.$$
For the compact form $\mathfrak{g}^c_2$, we can consider it as the space of pairs \[(iS + \star^{-1} w, y+i x) \in \mathfrak{su}(3) \oplus \mathbb{C}^3, \quad x, y, w \in \mathbb{R}^3, \, S \in \mathfrak{s}.\] The Lie brackets are given by
$$\left[ \begin{pmatrix}  i S_1 + \star^{-1} w_1 \\  0 \end{pmatrix} , \begin{pmatrix}  i S_2 + \star^{-1} w_2 \\ 0 \end{pmatrix} \right] = \begin{pmatrix} [iS_1+ \star^{-1} w_1, i S_2 + \star^{-1} w_2]  \\ 0
\end{pmatrix},$$
$$\left[ \begin{pmatrix}  i S + \star^{-1} w \\  0 \end{pmatrix} , \begin{pmatrix}  0 \\ y +i x \end{pmatrix} \right] = \begin{pmatrix} 0 \\
(iS+ \star^{-1} w)(y+ix) 
\end{pmatrix},$$
\begin{align*} &\left[ \begin{pmatrix} 0 \\  y_1 +i x_1 \end{pmatrix} , \begin{pmatrix} 0 \\ y_2 +i x_2 \end{pmatrix} \right] \\
& = \begin{pmatrix}  3 (y_2 + i x_2)(y_1-i x_1)^t - 3 (y_1+ i x_1)(y_2 -i x_2)^t  +2i (\langle x_1, y_2 \rangle - \langle y_1, x_2 \rangle)I \rangle \\
2 (y_1-i x_1) \times (y_1 - i x_2) \end{pmatrix} \\
& = \begin{pmatrix}  6i x_2 \otimes y_1 - 6i x_1 \otimes y_2  + x_1 \times x_2 + y_1 \times y_2 \\
2 y_1 \times y_2 - 2x_1 \times x_2 - 2i (x_1 \times y_2 + y_1 \times x_2 )\end{pmatrix},\end{align*}
where the symbol $\times$ denotes the cross product on $\mathbb{R}^3$ extended by linearity. Using these models, we have the following explicit description.
\begin{theorem} \label{th:A33}
Let $(M, H, g)$ be a sub-Riemannian model space with nilpotentization isometric to $\mathsf{A}_{3,3}$. Write $M = G/K$, where $G = \Isom(M)$ and $K$ is the isotropy group of some point. Let $\mathfrak{g}$ and $\mathfrak{k}$ be Lie algebras of respectively $G$ and~$K$. Then $(M, H,g)$ is isometric to one of the following cases after an appropriate scaling of the metric.
\begin{enumerate}[\rm (i)]
\item The sub-Riemannian model space $M$ is isometric to $\mathsf{A}_{3,3}$,
\item The Lie algebra $\mathfrak{g}$ is isomorphic to $\mathfrak{g}_2^s$ and
$$\mathfrak{k} = \left\{ C_w =\begin{pmatrix} 0 & 0 & 0 \\ 0 & \star^{-1} w & 0\\ 0 & 0 & \star^{-1} w \end{pmatrix} : w \in \mathbb{R}^3 \right\}.$$
The sub-Riemannian structure $(H,g)$ is the projection of a sub-Riemannian structure on $G$ defined by left translation of
$$\mathfrak{p} = \left\{ B_x = \begin{pmatrix} 0 & 2 x^t & -2x^t \\ x & 0 & -\star^{-1} x \\ -x & \star^{-1} x & 0 \end{pmatrix} : x \in \mathbb{R}^3 \right\},$$
with inner product $\langle B_x, B_y \rangle = \langle x, y\rangle$.
\item The Lie algebra $\mathfrak{g}$ is isomorphic to $\mathfrak{g}_2^c$ and
$$\mathfrak{k} = \left\{ C_w = (\star^{-1} w, 0) \in \mathfrak{su}(3) \times \mathbb{C}^3 : w \in \mathbb{R}^3 \right\}.$$
The sub-Riemannian structure $(H,g)$ is the projection of a sub-Riemannian structure on $G$ defined by left translation of
$$\mathfrak{p} = \left\{ B_x = (0, ix)\in \mathfrak{su}(3) \times \mathbb{C}^3 : x \in \mathbb{R}^3 \right\},$$
with inner product $\langle B_x, B_y \rangle = \langle x, y\rangle$.
\end{enumerate}
\end{theorem}
We will use the rest of this section to prove this result.

\subsection{Model Spaces as Quotients of $\mathrm{G}_2$}
We will first show that the spaces in Theorem~\ref{th:A33} are indeed model spaces.
\begin{lemma}
Let $(M,H,g)$ be a sub-Riemannian manifold on the form {\rm (ii)} or {\rm (iii)} in Theorem~\ref{th:A33}. Then it is a model space whose nilpotentization is isometric to $\mathsf{A}_{3,3}$.
\end{lemma}
\begin{proof}
For both cases, the action of $\Ort(3)$ on the Lie algebra $\mathfrak{g}$ is determined by being a Lie algebra homomorphism and satisfying
$$q B_x =  B_{qx}, \qquad qC_w =(\det q) C_{qw}.$$
For the case of $\mathfrak{g}_2^s$, we have the induced action such that if we write
\begin{equation} \label{G2Scoordinates}
\begin{pmatrix} x \\ y \\S \\w \end{pmatrix}' = \begin{pmatrix} 0 & -2 y^t & -2x^t \\ x & S + \star^{-1} w & \star^{-1} y \\ y & \star^{-1} x & - S + \star^{-1} w \end{pmatrix}\end{equation}
then
\begin{equation} \label{Qaction1} q\cdot \begin{pmatrix} x \\ y \\ S \\ w \end{pmatrix}' = \begin{pmatrix} \frac{1}{2} q(x-y) + \frac{1}{2} (\det q) q(x+y)\\  \frac{1}{2} q(y-x) + \frac{1}{2} (\det q) q(x+y) \\ (\det q) qSq^t \\ (\det q) qw \end{pmatrix}' , \qquad q \in \Ort(3).\end{equation}
One can verify that the bracket
\begin{align*}
\left[ \begin{pmatrix} x_1 \\ y_1 \\ S_1 \\ w_1 \end{pmatrix}' , \begin{pmatrix} x_2 \\ y_2 \\ S_2 \\ w_2 \end{pmatrix}' \right] = 
\begin{pmatrix} 2 y_1 \times y_2+ S_1 x_2 - S_2 x_1 + w_1 \times x_2 + x_1 \times w_2  \\ 2 x_1 \times x_2 + S_2 y_1 - S_1 y_2 + w_1 \times y_2 + y_1 \times w_2 \\ 3 x_2 \odot y_1- 3 x_1 \odot y_2+[\star^{-1} w_1, S_2] + [S_1, \star^{-1} w_2] \\  \frac{3}{2} x_1 \times y_2 + \frac{3}{2} y_1 \times x_2 +\star [S_1, S_2] + w_1 \times w_2 \end{pmatrix}'
\end{align*}
is invariant under this action.

Similarly, for $\mathfrak{g}_2^c$, for each $q \in \Ort(3)$, we get a Lie algebra homomorphism on $\mathfrak{su}(3) \times \mathbb{C}^3$ with its given brackets, by
\begin{equation} \label{Qaction2} q \cdot \begin{pmatrix} iS + i \star^{-1} w \\ y +ix \end{pmatrix} = \begin{pmatrix} (\det q) ( i qS q^t + \star^{-1} q w) \\ (\det q) qy + i q x \end{pmatrix}.\end{equation}

We can check from the Lie brackets that in both cases, we have the desired nilpotentization.
\end{proof}

\subsection{Classifying Result}
We will now show that the number of model spaces with nilpotentization $\mathsf{A}_{3,3}$ is, up to a scaling of the metric, finite.
\begin{lemma}
\label{A_33_Classification}
Let $(M,H,g)$ be a sub-Riemannian model space with nilpotentization isometric to $\mathsf{A}_{3,3}$. Then $M$ is isometric to $\mathsf{A}_{3,3}(\kappa)$ for $\kappa \in \mathbb{R}$, where $\mathsf{A}_{3,3}(\kappa)$ is a model space with the following description: The Lie algebra $\mathfrak{g}$ of $\Isom\left(\mathsf{A}_{3,3}(\kappa)\right)$ has the identification 
\begin{equation*} \mathfrak{g} \simeq \mathbb{R}^3 \oplus \bar{\mathbb{R}}^3 \oplus \bar{\mathfrak{s}} \oplus \bar{\mathbb{R}}^3, \end{equation*}
where the last $\mathbb{R}^3$-term is identified with the Lie algebra of the isotropy group at $x$. The Lie bracket between elements in $\mathfrak{g}$ is given by
\begin{align*} \left [ \begin{pmatrix} x_1 \\ y_1 \\ S_1 \\ w_1 \end{pmatrix}, \begin{pmatrix} x_2 \\ y_2 \\ S_2 \\ w_2 \end{pmatrix} \right ]
& =
\begin{pmatrix}
0 \\ x_1 \times x_2  \\ x_1 \odot y_2 - y_1 \odot x_2  \\ 0 \end{pmatrix} + \begin{pmatrix}
w_1 \times x_2 + x_1 \times w_2 \\   w_1 \times y_2 + y_1 \times w_2 \\ [\star^{-1} w_1, S_2] + [S_1,\star^{-1} w_2]  \\  w_1 \times w_2
\end{pmatrix} \\
&\qquad + \kappa \begin{pmatrix}
7(x_1 \times y_2 + y_1 \times x_2) \\  2 y_1 \times y_2 +6(S_1 x_2 - S_2 x_1)  \\ 7([ \star^{-1} y_1,S_2] + [S_1, \star^{-1} y_2]) \\ 0 \end{pmatrix}
\\
& \qquad
+ \kappa^2 \begin{pmatrix}
24(S_2 y_1 - S_1 y_2)  \\  0 \\ 0 \\ 15 y_1 \times y_2 + 18(S_2 x_1 - S_1 x_2) \end{pmatrix} - 144 \kappa^3  \begin{pmatrix}
0 \\  0 \\ 0 \\  \star [S_1,S_2]  \end{pmatrix}
.\end{align*}
The spaces $\mathsf{A}_{3,3}(\kappa)$ for $\kappa \in \mathbb{R}$ form a non-isometric family of sub-Riemannian model spaces, implying that $(M,H,g)$ is uniquely determined by the constant $\kappa$. Finally, for any positive constant $\lambda > 0$, we note that \[\lambda \cdot \mathsf{A}_{3,3}(\kappa) = \mathsf{A}_{3,3}(\kappa/\lambda^2).\]
\end{lemma}

\begin{proof}
We will set up a correspondence between the model space structure on $(M,H,g)$ and a decomposition of the Lie algebra $\mathfrak{g}$ of its isometry group $G: = \Isom(M)$. We will use the results developed in Section~\ref{sec:Equivariant} to restrict the possible Lie algebra structures, and finally reduce the possible options further using the Jacobi identity. Lastly, we will show that our construction determines the model spaces with nilpotentization $\mathsf{A}_{3,3}$ uniquely.

Fix a point $x \in M$ and let $K := K_{x}$ be the isotropy group corresponding to $x$. The notation~$\mathfrak{k}$ indicates as usual the Lie algebras of $K$ while $\pi:G \to M$ denotes the projection sending $\varphi$ to $\pi(\varphi) := \varphi(x)$. We let $\sigma:M \to M$ be the unique isometry such that \[d\sigma |_{H_{x}} = -\id_{H_{x}}.\]
Since $\sigma^2 = \id$, we can consider the eigenvalue decomposition relative to the map $\Ad(\sigma)$, 
\[\mathfrak{g} = \mathfrak{g}^{+} \oplus \mathfrak{g}^{-}, \qquad \mathfrak{k} \subset \mathfrak{g}^{+}.\]
We will now use the canonical partial connection $\nabla$ corresponding to $H$ in Proposition \ref{canonical_connection_theorem} (a) to decompose of $\mathfrak{g}^{+}$ and $\mathfrak{g}^{-}$ further. Let $\mathfrak{a}_1 \subset \mathfrak{g}^{-}$ denote the subspace corresponding to the canonical partial connection $\nabla$ as in Proposition \ref{canonical_connection_theorem} (b). As $\mathfrak{a}_1$ is invariant under the action of the compact group $K$, we have that there exists an invariant complement $\mathfrak{a}_{3}\subset \mathfrak{g}^{-}$ of $\mathfrak{a}_1$. We define \[\mathfrak{a}_{2} = [\mathfrak{a}_1,\mathfrak{a}_{1}] \subset \mathfrak{g}^{+}.\]
Notice that $H_{x}^{2} \subset d_1\pi(\mathfrak{a}_1 + \mathfrak{a}_2)$.
However, the reverse inclusion also holds because the nilpotentization of $M$ is isometric to $\mathsf{A}_{3,3}$. This implies that $\mathfrak{a}_2$ is transverse to $\mathfrak{k} = \textrm{ker}(d_1\pi)$ inside $\mathfrak{g}^{+}$. Summarizing, we have the decomposition 
\begin{equation}
\label{Lie_algebra_decomposition}
    \mathfrak{g} = \mathfrak{a}_1 \oplus \mathfrak{a}_2 \oplus \mathfrak{a}_3 \oplus \mathfrak{k}.
\end{equation}
\par Let us fix an orthonormal basis for $H_{x}$ by choosing a linear isometry $\phi:\mathbb{R}^3 \to H_{x}$. Then the map \[\phi^{-1} \circ d_1\pi |_{\mathfrak{a}_1}:\mathfrak{a}_1 \longrightarrow \mathbb{R}^3\] identifies $\mathfrak{a}_1$ and $\mathfrak{a}_2$ with $\mathbb{R}^3$ as vector spaces. Moreover, this induces an identification between $\mathfrak{a}_3$ and $\mathfrak{s}$ as vector spaces. Recall that any isometry $\varphi \in K$ is uniquely determined by its differential at a single point by Proposition \ref{Regularity_of_isometries}. We hence get an identification between $K$ and $\Ort(3)$ as Lie groups since $(M,H,g)$ is a model space. This identifies the Lie algebra $\mathfrak{k}$ of $K$ with the Lie algebra $\mathfrak{so}(3)$ of $O(3)$. An element in $\mathfrak{g}$ will be denoted by $(x,y,S,w)$ according to the decomposition (\ref{Lie_algebra_decomposition}) with $x,y,w \in \mathbb{R}^3$ and $S \in \mathfrak{s}$. We will use the following properties.
\begin{enumerate}[\rm (I)]
\item The Lie brackets are invariant under $\Ort(3)$, that is, for any $a \in \Ort(3)$ we have 
 \begin{align*}a \cdot [(x_1,y_1,S_1,w_1),(x_2,y_2,S_2,w_2)] = [a \cdot (x_1,y_1,S_1,w_1), a \cdot (x_2,y_2,S_2,w_2)]. \end{align*}
 \item Since the Lie bracket of $ \mathfrak{k}$ when identified with $\mathbb{R}^3$ is the cross product, we have \begin{align*}[(0,0,0,w_1), (0,0,0,w_2)] = (0,0,0,w_1 \times w_2). \end{align*}
\item By construction, we have \begin{align*}[\mathfrak{a}_i,\mathfrak{k}] \subset \mathfrak{a}_i, \quad i=1,2,3. \end{align*}
 \item Since the nilpotentization $\textrm{Nil}(M)$ is isometric to $\mathsf{A}_{3,3}$, it follows that
 \begin{align*}[(x_1,0,0,0), (x_2,0,0,0)] & = (0,x_1 \times x_2,0,0), \\  \pr_{\mathfrak{a}_3}[(x,0,0,0), (0,y,0,0)] &= (0,0,x \odot y,0),\end{align*} where $\odot$ is defined in (\ref{symmetrization_product}).
\end{enumerate}
Using the above properties and Lemma \ref{All_The_Invariant_Maps}, we deduce that the Lie bracket between the elements $(x_1,y_1,S_1,w_1)$ and $(x_2,y_2,S_2,w_2)$ has the form
\begin{align*}
    \small \begin{pmatrix}
c_5(x_1 \times y_2 + y_1 \times x_2) + c_6(S_2 y_1 - S_1 y_2) + x_1 \times w_2 + w_1 \times x_2 \\ x_1 \times x_2 + c_1y_1 \times y_2 + c_3[S_1,S_2] + c_8(S_2 x_1 - S_1 x_2) + y_1 \times w_2 + w_1 \times y_2 \\ x_1 \odot y_2 - y_1 \odot x_2 + c_7\left([y_1,S_2] - [y_2,S_1]\right) + [S_1,w_2] - [S_2,w_1] \\ c_2y_1 \times y_2 + c_4[S_1,S_2] + c_{9}(S_2 x_1 - S_1 x_2) + w_1 \times w_2
\end{pmatrix}.
\end{align*}

We will introduce further restriction of the constants $c_1,\dots,c_9$ by imposing that the Jacobi identity holds for carefully selected basis elements. We denote the standard basis in $\mathbb{R}^3$ by $e_1,e_2,e_3$ and we will employ the symmetric matrices $S_{12},S_{13},S_{23}$ and the diagonal matrices $D_1,D_2,D_3$ defined in (\ref{skew_symmetric matricies}).
To obtain some simple relations we let \[v_1 = \begin{pmatrix} e_1 \\ 0 \\ 0 \\ 0 \end{pmatrix}, \quad v_2 = \begin{pmatrix} e_2 \\ 0 \\ 0 \\ 0 \end{pmatrix}, \quad v_3 = \begin{pmatrix} 0 \\ e_1 \\ 0 \\ 0 \end{pmatrix}.\]
Then the Jacobi identity becomes
\[\left [\begin{pmatrix} e_1 \\ 0 \\ 0 \\ 0 \end{pmatrix}, \begin{pmatrix} -c_5 e_3 \\ 0 \\ \frac{1}{2} S_{12} \\ 0 \end{pmatrix} \right ] - \left [\begin{pmatrix} e_2 \\ 0 \\ 0 \\ 0 \end{pmatrix}, \begin{pmatrix} 0 \\ 0 \\ \frac{1}{3}(2 D_1 - D_2 - D_3) \\ 0 \end{pmatrix} \right ] + \left [\begin{pmatrix} 0 \\ e_1 \\ 0 \\ 0 \end{pmatrix}, \begin{pmatrix} 0 \\ e_3 \\ 0 \\ 0 \end{pmatrix} \right ] = \begin{pmatrix} 0 \\ 0 \\ 0 \\ 0 \end{pmatrix}.\] By computing the final brackets we acquire the equations
\begin{align*}
   \frac{5}{6}c_9 & = c_2, \tag{E1}\label{E1} \\
   c_5 + \frac{5}{6}c_8 & = c_1. \tag{E2}\label{E2}
\end{align*}
As the next four computations are of a similar nature, we will provide fewer details. Let \[v_1 = \begin{pmatrix} 0 \\ e_1 \\ 0 \\ 0 \end{pmatrix}, \quad v_2 = \begin{pmatrix} 0 \\ e_2 \\ 0 \\ 0 \end{pmatrix}, \quad v_3 = \begin{pmatrix} e_1 \\ 0 \\ 0 \\ 0 \end{pmatrix}.\] From the Jacobi identity we extract the equation $3c_7 = c_5 + c_1.$ Using Equation (\ref{E2}) we can rewrite this as 
\begin{align*}
    3c_7 = 2c_1 -\frac{5}{6}e_8. \tag{E3}\label{E3} 
\end{align*}
Next we start to involve $\mathfrak{s}$ and consider \[v_1 = \begin{pmatrix} e_1 \\ 0 \\ 0 \\ 0 \end{pmatrix}, \quad v_2 = \begin{pmatrix} 0 \\ e_1 \\ 0 \\ 0 \end{pmatrix}, \quad v_3 = \begin{pmatrix} 0 \\ 0 \\ S_{12} \\ 0 \end{pmatrix}.\] This time the equations we get are 
\begin{align*}
   c_3 & = c_1 c_8 + c_9 - c_6 -c_7 c_8, \tag{E4}\label{E4} \\
   c_4 & = c_2 c_8 - c_7 c_9. \tag{E5}\label{E5}
\end{align*}
By letting \[v_1 = \begin{pmatrix} e_1 \\ 0 \\ 0 \\ 0 \end{pmatrix}, \quad v_2 = \begin{pmatrix} e_2 \\ 0 \\ 0 \\ 0 \end{pmatrix}, \quad v_3 = \begin{pmatrix} 0 \\ 0 \\ S_{13} \\ 0 \end{pmatrix},\] we obtain the equations 
\begin{align*}
   c_6 & = -c_5c_8 -c_9, \tag{E6}\label{E6} \\
   c_7 & = -\frac{1}{2}c_8. \tag{E7}\label{E7}
\end{align*}
Notice that we can use equations (\ref{E1}) - (\ref{E7}) to write all the coefficients in terms of $c_1$ and $c_2$. Finally, to obtain a dependence between $c_1$ and $c_2$ we consider \[v_1 = \begin{pmatrix} 0 \\ e_1 \\ 0 \\ 0 \end{pmatrix}, \quad v_2 = \begin{pmatrix} 0 \\ e_2 \\ 0 \\ 0 \end{pmatrix}, \quad v_3 = \begin{pmatrix} 0 \\ 0 \\ S_{13} \\ 0 \end{pmatrix}.\] The Jacobi identity gives the equation
\begin{align*}
    c_2 = c_7^2 - c_1c_7 + \frac{1}{2}c_6. \tag{E8}\label{E8}
\end{align*}
\par It is now straightforward to use equations (\ref{E1}) - (\ref{E8}) to write the coefficients $c_1,\dots,c_9$ in terms of a single coefficient. We rename $7\kappa := c_5$ and get after some straightforward manipulations the equations 
\begin{align*}
       c_1 & = 2\kappa, \quad
       c_2 = 15\kappa^2, \quad
       c_3 = 0, \quad 
       c_4 = -144\kappa^3, \\
       c_5 = 7\kappa &, \quad
       c_6 = 24\kappa^2, \quad
       c_7 = 3 \kappa, \quad
       c_8 = -6 \kappa, \quad
       c_9 = 18\kappa^2.
\end{align*}
Thus we can conclude that the Lie bracket between the elements $(x_1,y_1,S_1,w_1)$ and $(x_2,y_2,S_2,w_2)$ actually has the form
\[\small \begin{pmatrix}
7\kappa(x_1 \times y_2 + y_1 \times x_2) + 24\kappa^2(S_2 y_1 - S_1 y_2) + x_1 \times w_2 + w_1 \times x_2 \\ x_1 \times x_2 + 2\kappa y_1 \times y_2 -6\kappa(S_2 x_1 - S_1 x_2) + y_1 \times w_2 + w_1 \times y_2 \\ x_1 \odot y_2 - y_1 \odot x_2 + 3\kappa([\star^{-1} y_1,A_2] - [\star^{-1} y_2,A_1]) + [S_1,w_2] - [S_2,w_1] \\ 15\kappa^2y_1 \times y_2 -144\kappa^3 \star [S_1,S_2] + 18\kappa^2(S_2 x_1 - S_1 x_2) + w_1 \times w_2
\end{pmatrix}.\] 
To see that this in fact satisfies the Jacobi identity one simply has to observe that everything cancels when expanding the identity. 

We will now show that the constant $\kappa$ uniquely determines the model spaces with nilpotentization $\mathsf{A}_{3,3}$. For $i=1,2$, let us use the temporary notation $\mathfrak{g}(\kappa_i)$ for the Lie algebra of $G^{(i)} := \textrm{Isom}(M^{(i)})$, where $(M^{(i)},H^{(i)},g^{(i)})$ are model spaces with nilpotentization $\mathsf{A}_{3,3}$ and $i = 1,2$. Assume there exists an isometry $\Phi:M^{(1)} \to M^{(2)}$. The discussion in Section \ref{sec: Partial_Connections_and_Horizontal_Holonomy} implies that we get an induced Lie algebra morphism
\[\psi := d\overline{\Phi}:\mathfrak{g}(\kappa_1) = \mathfrak{a}_1 \oplus \mathfrak{a}_2 \oplus \mathfrak{a}_3 \oplus \mathfrak{k} \longrightarrow \mathfrak{g}(\kappa_2) = \widetilde{\mathfrak{a}_1} \oplus \widetilde{\mathfrak{a}_2} \oplus \widetilde{\mathfrak{a}_3} \oplus \widetilde{\mathfrak{k}}\] that maps $\mathfrak{a}_1$ isometrically onto $\widetilde{\mathfrak{a}_1}$ and $\psi(\mathfrak{k}) = \widetilde{\mathfrak{k}}$. As $[\mathfrak{a}_1,\mathfrak{a}_1] = \mathfrak{a}_2$ and similarly for $\widetilde{\mathfrak{a}_2}$, we get that after identifying $\mathfrak{a}_1$ with $\widetilde{\mathfrak{a}_1}$ through $\psi$ that $\mathfrak{a}_2$ gets identified with $\widetilde{\mathfrak{a}_2}$ as well. However, we can not conclude that $\mathfrak{a}_3$ is mapped onto $\widetilde{\mathfrak{a}_3}$. Nevertheless, we know that $\psi$ will map $[\mathfrak{a}_1,\mathfrak{a}_2]$ into $[\widetilde{\mathfrak{a}_1},\widetilde{\mathfrak{a}_2}]$. As $\mathfrak{a}_3 \subset [\mathfrak{a}_1,\mathfrak{a}_2]$ and similarly for $\widetilde{\mathfrak{a}_3}$, we can at least ensure that
\[\psi |_{\mathfrak{a}_3}:\mathfrak{a}_3 \longrightarrow [\widetilde{\mathfrak{a}_1},\widetilde{\mathfrak{a}_2}] \subset \widetilde{\mathfrak{a}_1} \oplus \widetilde{\mathfrak{a}_3}.\]

If $S \in \mathfrak{a}_3$ we will use the notation $\psi(S)_{j}$ according to the decomposition of $\psi(\mathfrak{a}_3) \subset \widetilde{\mathfrak{a}_1} \oplus \widetilde{\mathfrak{a}_3}$, for $j = 1,3.$ For $x \in \mathfrak{a}_1$ and $y \in \mathfrak{a}_2$, we have 
\begin{align*}
\psi\left(\left[ \begin{pmatrix} x \\ 0 \\ 0 \\ 0 \end{pmatrix} , \begin{pmatrix} 0 \\ y \\ 0 \\ 0 \end{pmatrix} \right]\right) = \psi\begin{pmatrix}7\kappa_1 x \times y \\ 0 \\ x \odot y \\ 0 \end{pmatrix} = \begin{pmatrix}7\kappa_1 \psi(x) \times \psi(y) + \psi\left(x \odot y \right)_{1} \\ 0 \\ \psi\left(x \odot y \right)_{3} \\ 0 \end{pmatrix}.
\end{align*}
On the other hand, we also have that
\[[\psi(x),\psi(y)] = \begin{pmatrix} 7\kappa_2 \psi(x) \times \psi(y) \\ 0 \\ \psi(x) \odot \psi(y) \\ 0\end{pmatrix},\] where we used that $\psi$ sends $\mathfrak{a}_2$ onto $\widetilde{\mathfrak{a}_2}$. Rewriting the first row gives 
\begin{equation}
\label{minor_equation}
(\kappa_{2}-\kappa_{1})\psi(x) \times \psi(y) = \psi\left(x \odot y\right)_{1}.
\end{equation}
However, notice that the left hand side of Equation (\ref{minor_equation}) is skew-symmetric while the right hand side is symmetric. Thus both sides are identically zero and it follows that $\kappa_1 \neq \kappa_2$ from picking $x \in \mathfrak{a}_1$ and $y \in \mathfrak{a}_2$ such that $\psi(x) \times \psi(y) \neq 0$. Hence different values of $\kappa \in \mathbb{R}$ parametrize a non-isometric family $\mathsf{A}_{3,3}(\kappa)$ of sub-Riemannian model spaces with nilpotentization $\mathsf{A}_{3,3}$. Through the construction presented, it is clear that every model space with nilpotentization $\mathsf{A}_{3,3}$ have the form presented in Lemma \ref{A_33_Classification}.
\end{proof}

\subsection{Proof of Theorem~\ref{th:A33}}
In Lemma~\ref{A_33_Classification} we showed that all model spaces $(M, H, g)$ with nilpotentization isometric to $\mathsf{A}_{3,3}$ are, up to a scaling of the metric, isometric to $\mathsf{A}_{3,3}$, $\mathsf{A}_{3,3}(-1)$ or $\mathsf{A}_{3,3}(1)$. Since our list in Theorem~\ref{th:A33} contains three examples that can not be scaled to be equal, they have to be in a one-to-one correspondence. 

For an explicit description, let $\mathfrak{g}(\kappa)$ denote the Lie algebra of $\textrm{Isom}\left(\mathsf{A}_{3,3}(\kappa)\right)$. For $\kappa > 0$ with $\sqrt{\kappa} =a$, we have an isomorphism $\varphi: \mathfrak{g}(\kappa) \to \mathfrak{g}_2^c$ given by
$$\varphi \begin{pmatrix} x \\ y \\ S \\ w \end{pmatrix} = \begin{pmatrix} 12 a^3 i S+ \star^{-1} \left(w + 3a^2 y\right) \\ -2 a^2 y + i ax 
\end{pmatrix} .$$
Similarly, for $\kappa < 0$ with $\kappa = - a^2$, we have an isomorphism $\varphi: \mathfrak{g}(\kappa) \to \mathfrak{g}_2^s$. In the coordinates \eqref{G2Scoordinates}, this isomorphism is given by
\begin{align*}
\varphi \begin{pmatrix} x \\ y \\ S \\ w \end{pmatrix} =
\begin{pmatrix}
ax + 2a^2 y\\ -ax +2a^2 \\ -12 a^3 S \\ w - 3a^2 y
\end{pmatrix} '. 
\end{align*}\qed

\section{Model spaces with nilpotentization $\mathsf{F}[3,3]$}
\label{sec:Section (3,6,14)-classification}
We will now provide the final piece of the classification which turns out to be the most surprising one as well. This result should be compared with the known results on step two model spaces briefly discussed in the beginning of Section \ref{sec: Model_spaces_classification}.

\begin{theorem}
\label{nilpotent_classification}
Let $(M,H,g)$ be a sub-Riemannian model space whose nilpotentization is isometric to the free nilpotent Lie group $\mathsf{F}[3,3]$. Then $M$ is isometric to $\mathsf{F}[3,3]$. 
\end{theorem}

\begin{proof}
It will be apparent that the proof has a similar structure as the proof of Lemma \ref{A_33_Classification}, hence we will provide fewer details.
Let $(M,H,g)$ denote a model space such that its nilpotentization $\textrm{Nil}(M)$ is isometric to $\mathsf{F}[3,3]$. Similarly as before, we get an eigenvalue decomposition $\mathfrak{g} = \mathfrak{g}^{+} \oplus \mathfrak{g}^{-}$ of the Lie algebra of the isometry group. The canonical partial connection $\nabla$ is again used to obtain a subspace $\mathfrak{a}_1 \subset \mathfrak{g}$ that is invariant under the action of the isotropy group $K$. We define $\mathfrak{a}_2 = [\mathfrak{a}_1,\mathfrak{a}_1]$ and note that $\mathfrak{a}_2$ is transverse to $\mathfrak{a}_1$ since $\mathfrak{a}_1 \subset \mathfrak{g}^{-}$ while $\mathfrak{a}_2 \subset \mathfrak{g}^{+}$. Moreover, the argument presented in the proof of Lemma \ref{A_33_Classification} carries over to show that $\mathfrak{a}_2$ is transverse to $\mathfrak{k}$. Define $\widetilde{\mathfrak{a}_3} = [\mathfrak{a}_1,\mathfrak{a}_2] \subset \mathfrak{g}^{-}$. Then $\widetilde{\mathfrak{a}_3}$ is eight-dimensional since 
\begin{equation}
\label{equation_1}
    d_1\pi(\mathfrak{a}_1 + \mathfrak{a}_2 + \widetilde{\mathfrak{a}_3}) = T_{p}M.
\end{equation} 
The subspace $\widetilde{\mathfrak{a}_3}$ is clearly transverse to both $\mathfrak{a}_2$ and $\mathfrak{k}$ due to the eigenvalue decomposition $\mathfrak{g} = \mathfrak{g}^{+} \oplus \mathfrak{g}^{-}$. Moreover, a nonempty intersection of $\mathfrak{a}_1$ and $\widetilde{\mathfrak{a}_3}$ would contradict (\ref{equation_1}). We emphasize that this argument is only valid because $\Nil(M)$ is isometric to $\mathsf{F}[3,3]$ and should be compared with the different strategy used in the proof of Lemma \ref{A_33_Classification}. \par We get an identification between $K$ and $\Ort(3)$ as Lie groups by fixing an orthonormal basis for $H_{x}$. This induces identifications 
\[\mathfrak{a}_1 \simeq \mathfrak{f}_1 \simeq \mathbb{R}^3, \quad \mathfrak{a}_2 \simeq \mathfrak{f}_2 \simeq \bar{\mathbb{R}}^3, \quad \widetilde{\mathfrak{a}_3} \simeq \mathfrak{f}_{3}\]
as representations, where $\mathfrak{f}_i$ denotes the $i$'th layer of the free nilpotent Lie algebra $\mathsf{f}[3,3]$. Finally, we use the concrete description of the action on $\mathfrak{f}_3$ given prior to Theorem \ref{Classification_Carnot_Spaces} to decompose $\widetilde{\mathfrak{a}_3} = \mathfrak{a}_3 \oplus \mathfrak{a}_4$ as representations, where $\mathfrak{a}_3 \simeq \bar{\mathfrak{s}}$ and $\mathfrak{a}_4 \simeq \mathfrak{so}(3) \simeq \bar{\mathbb{R}}^3$. We use $\star: \mathfrak{so}(3) \to \mathbb{R}^3$ for the latter identification. Summarizing these properties gives us the identification
\begin{equation}
\label{g_decomposition_N33}
\mathfrak{g} = \mathfrak{a}_1 \oplus \mathfrak{a}_2 \oplus \mathfrak{a}_3 \oplus \mathfrak{a}_4 \oplus \mathfrak{k} \simeq \mathbb{R}^3 \oplus \bar{\mathbb{R}}^3 \oplus \bar{\mathfrak{s}} \oplus \mathbb{R}^3 \oplus \bar{\mathbb{R}}^3.
\end{equation}
\par We will denote an arbitrary element in $\mathfrak{g}$ by $(x,y,S,z,w)$ according to the decomposition (\ref{g_decomposition_N33}). The Lie bracket of $\mathfrak{g}$ satisfies the properties (I) - (III) presented in the proof of Lemma \ref{A_33_Classification}, along with the following modification of (IV):
\begin{enumerate}
\item[(IV')] Since the nilpotentization of $M$ is isometric to $\mathsf{F}[3,3]$, it follows that
 \begin{align*}
 [(x_1,0,0,0,0), (x_2,0,0,0,0)] & = (0,x_1 \times x_2,0,0,0), \\ 
 [(x,0,0,0,0), (0,y,0,0,0)] & = (0,0,x \odot y, x \times y,0).
 \end{align*}
\end{enumerate}
Using these properties and Lemma~\ref{All_The_Invariant_Maps}, we have that the Lie bracket between two elements $(x_1,y_1,S_1,z_1,w_1)$ and $(x_2,y_2,S_2,z_2,w_2)$ has the form
\begin{align*}
& \begin{pmatrix}
 0 \\ x_1 \times x_2 + b_1y_1 \times y_2 + b_2 \star [S_1,S_2] + b_3z_1 \times z_2 \\  0 \\ 0 \\ f_1y_1 \times y_2 + f_2[S_1,S_2] + f_3 z_1 \times z_2 + w_1 \times w_2
\end{pmatrix}  \\ & +
\begin{pmatrix}
 a_1S_{2}y_1 + a_2y_1 \times z_2 + x_1 \times w_2 \\ b_4 S_{1}z_2 + b_5S_{2}x_1 + b_6x_1 \times z_2 + y_1 \times w_2 \\  x_1 \odot y_2 + c_1[\star^{-1} y_1,S_2] + c_2 y_1 \odot z_2 + [S_1, \star^{-1} w_2] \\ x_1 \times y_2 + d_1 S_{2}y_1 + d_2y_1 \times z_2 + z_1 \times w_2 \\ f_4S_{1}z_2 + f_5S_{2}x_1 + f_6x_1 \times z_2
\end{pmatrix}  \\
& - \begin{pmatrix} a_1 S_1 y_2 + a_2y_2 \times z_1 + x_2 \times w_1\\ b_4 S_2z_1 + b_5 S_1x_2 + b_6x_2 \times z_1 + y_2 \times w_1\\ y_1 \odot x_2 + c_1[\star^{-1} y_2,S_1] + c_2 z_1 \odot y_2 + [S_2, \star^{-1} w_1]\\ x_2 \times y_1 + d_1S_1 y_2 + d_2y_2 \times z_1 + z_2 \times w_1\\ f_4S_2z_1 + f_5S_1x_2 + f_6x_2 \times z_1 \end{pmatrix},
\end{align*}
for some constants $a_1$, $a_2$, $b_1$, $b_2$, $b_3$, $b_4$, $b_5$, $b_6$, $c_1$, $c_2$, $d_1$, $d_2$, $f_1$, $f_2$, $f_3$, $f_4$, $f_5$, $f_6$. \par
We will now derive no less than eighteen equations by using the Jacobi identity. From the obtained equations we will show that all the constants present in the Lie bracket above are in fact zero. As their derivations are straightforward and already illustrated in the proof of Lemma \ref{A_33_Classification}, we will only give the result and not the explicit calculations. We obtain the equations
\begin{align}
   \textrm{(Eq1)} \quad b_1 & = b_6 + (5/6)b_5, \quad & \textrm{(Eq2)} \quad f_1 & = f_6 + (5/6)f_5, \notag \\ \textrm{(Eq3)} \quad b_1 & = 3c_1 + c_2, \quad &
   \textrm{(Eq4)} \quad f_2 & = -d_1b_4 - d_2b_2, \notag \\ \textrm{(Eq5)} \quad f_4 & = (1/2)c_2b_2 - c_1 b_4, \quad & \textrm{(Eq6)} \quad f_4 & = -b_5a_2, \notag \\
   \textrm{(Eq7)} \quad b_4 & = c_2b_5, \quad & \textrm{(Eq8)} \quad f_5 & = -d_2b_5 - b_4, \notag \\ \textrm{(Eq9)} \quad f_5 & = -a_1, \quad &
   \textrm{(Eq10)} \quad b_5 & = -2c_1, \notag \\ \textrm{(Eq11)} \quad b_5 & = -d_1, \quad & \textrm{(Eq12)} \quad b_2 & = -6(f_5 + c_1b_5), \notag\\
   \textrm{(Eq13)} \quad f_6 & = a_2, \quad & \textrm{(Eq14)} \quad b_6 & = -c_2, \notag \\ \textrm{(Eq15)} \quad b_6 & = d_2, \quad & 
   \textrm{(Eq16)} \quad f_3 & = a_2b_6, \notag \\ \textrm{(Eq17)} \quad b_3 & = -c_2b_6, \quad & \textrm{(Eq18)} \quad f_6 & = b_3 - d_2b_6. \notag
\end{align}

A careful look at the equations above reveals that all the constants have to be zero once we have showed that $c_2$ has to be zero. Writing out (Eq18) by applying (Eq17), (Eq15), and (Eq14) gives \[f_6 = b_3 - d_2b_6 = -c_2 b_6 - b_6^2 = b_6^2 - b_6^2 = 0.\] Together with (Eq13) this shows that $a_2 = 0$. Looking at (Eq6) shows now that $f_4 = 0$. The goal now is to show that \[ f_4 = \frac{72}{49}c_2^3,\] as this will force $c_2$ to be zero. Firstly, we have from equation (Eq8) that \[f_5 = -d_2 b_5 - b_4 = -b_6b_5 - c_2b_5 = -b_6 b_5 + b_6 b_5 = 0,\]
by using (Eq15), (Eq7), and (Eq14). Expressing $f_4$ with the help of (Eq5) and using (Eq12), (Eq7), and (Eq10) shows that 
\begin{align*}
f_4 = \frac{1}{2}c_2b_2 - c_1 b_4 & = \frac{1}{2}c_2 (-6c_1b_5) - c_1 (c_2b_5) \\ & = -4c_1c_2b_5 \\ & = 8c_1^2 c_2 = \frac{72}{49}c_2^3. 
\end{align*}
Thus all the constants are zero. Hence $\mathfrak{g}$ is isomorphic to the Lie algebra of $\textrm{Isom}(\mathsf{F}[3,3])$. 
\end{proof}

\begin{remark} \label{remark:G2SR}
Although there are no other model spaces with nilpotentization $\mathsf{F}[3,3]$, there are other spaces whose isometry group have maximal dimension. In fact, there are examples of this both in the compact and split realization of $\mathrm{G}_2$. We use the same notation as in Section~\ref{sec:(3,6,11)-classification}. \begin{enumerate}[$\bullet$]
    \item On the split real form $\mathrm{G}_2^s$ there is the left-invariant sub-Riemannian structure that is given at the identity by
$$A_x = \begin{pmatrix} 0 & 0 & -2x^t \\  x & 0 & 0 \\ 0 & \star^{-1} x & 0 \end{pmatrix} \in \mathfrak{g}_2^s, \qquad \langle A_x, A_y \rangle = \langle x, y \rangle, \qquad x,y \in \mathbb{R}^3.$$
\item On the compact form $\mathrm{G}^c_2$ there is the left-invariant sub-Riemannian structure that is given at the identity by
$$A_x = \begin{pmatrix} 0 , x+ix\end{pmatrix} \in \mathfrak{g}^c_2, \qquad \langle A_x, A_y \rangle = \langle x, y \rangle.$$
\end{enumerate}
The actions in \eqref{Qaction1} and \eqref{Qaction2} preserves the sub-Riemannian structures when $q \in \SO(3)$, but it can not be extended to an action of $\Ort(3)$. One should expect these to be constant curvature spaces in the sense of \cite{AMS17}.
\end{remark}



\bibliographystyle{habbrv}
\bibliography{main.bib}

\begin{thebibliography}{10}

\bibitem{AMS17}
D.~{Alekseevsky}, A.~{Medvedev}, and J.~{Slovak}.
\newblock {Constant Curvature Models in Sub-Riemannian Geometry}.
\newblock {\em arXiv e-prints}, page arXiv:1712.10278, Dec. 2017, 1712.10278.

\bibitem{BoMo09}
G.~Bor and R.~Montgomery.
\newblock {$G_2$} and the rolling distribution.
\newblock {\em Enseign. Math. (2)}, 55(1-2):157--196, 2009.

\bibitem{Smoothness_of_sub_riemannian_isometries}
L.~Capogna and E.~Le~Donne.
\newblock Smoothness of sub{R}iemannian isometries.
\newblock {\em Amer. J. Math.}, 138(5):1439--1454, 2016.

\bibitem{CGJK15}
Y.~{Chitour}, E.~{Grong}, F.~{Jean}, and P.~{Kokkonen}.
\newblock {Horizontal holonomy and foliated manifolds}.
\newblock {\em To appear in: Annales de l'Instittutt Fourier, ArXiv e-prints},
  Nov. 2015, 1511.05830.

\bibitem{Chow1940}
W.-L. Chow.
\newblock {\"U}ber systeme von liearren partiellen differentialgleichungen
  erster ordnung.
\newblock {\em Mathematische Annalen}, 117(1):98--105, Dec 1940.

\bibitem{Dra18}
C.~Draper~Fontanals.
\newblock Notes on {$G_2$}: the {L}ie algebra and the {L}ie group.
\newblock {\em Differential Geom. Appl.}, 57:23--74, 2018.

\bibitem{Polish_Metric_Spaces}
S.~Gao and A.~S. Kechris.
\newblock On the classification of {P}olish metric spaces up to isometry.
\newblock {\em Mem. Amer. Math. Soc.}, 161(766):viii+78, 2003.

\bibitem{sub_Riemannian_Model_Spaces}
E.~{Grong}.
\newblock {Model spaces in sub-Riemannian geometry}.
\newblock {\em To appear in: Communications in Analysis and Geometry, ArXiv
  e-prints}, Oct. 2016, 1610.07359.

\bibitem{Kobayashi_Numisu}
S.~Kobayashi and K.~Nomizu.
\newblock {\em Foundations of differential geometry. {V}ol. {I}}.
\newblock Wiley Classics Library. John Wiley \& Sons, Inc., New York, 1996.
\newblock Reprint of the 1963 original, A Wiley-Interscience Publication.

\bibitem{Montgomery}
R.~Montgomery.
\newblock {\em A tour of subriemannian geometries, their geodesics and
  applications}, volume~91 of {\em Mathematical Surveys and Monographs}.
\newblock American Mathematical Society, Providence, RI, 2002.

\bibitem{Semi-Riemannian}
B.~O'Neill.
\newblock {\em Semi-{R}iemannian geometry}, volume 103 of {\em Pure and Applied
  Mathematics}.
\newblock Academic Press, Inc. [Harcourt Brace Jovanovich, Publishers], New
  York, 1983.
\newblock With applications to relativity.

\bibitem{Peter_Petersen}
P.~Petersen.
\newblock {\em Riemannian Geometry}.
\newblock Graduate Texts in Mathematics. Springer International Publishing, 3rd
  edition, 2016.

\bibitem{Ras38}
P.~K. Rashevski\u\i.
\newblock On the connectability of two arbitrary points of a totally
  nonholonomic space by an admissible curve.
\newblock {\em Uchen. Zap. Mosk. Ped. Inst. Ser. Fiz.-Mat. Nauk}, 3(2):83--94,
  1938.

\end{thebibliography}

\end{document}